\documentclass[12pt]{article}
\usepackage{amsmath,amssymb,amsfonts,amsthm,amscd,url,xspace,graphicx,psfrag,pstricks}
\newtheorem{theorem}{Theorem}
\newtheorem{thm}[theorem]{Theorem}
\numberwithin{theorem}{section}

\newtheorem{lemma}[theorem]{Lemma}

\newtheorem{cor}[theorem]{Corollary}
\newtheorem{conj}[theorem]{Conjecture}

\newcommand{\Z}{\mathbb{Z}}
\newcommand{\R}{\mathbb{R}}
\newcommand{\C}{\mathbb{C}}
\newcommand{\G}{\mathcal{G}}
\newcommand{\T}{\mathbb{T}}
\newcommand{\No}{\mathcal{N}}

\newcommand{\be}{\begin{equation}}
\newcommand{\ee}{\end{equation}}
\newcommand{\SL}{\text{SL}_2(\C)}
\newcommand{\SU}{\text{SU}_2}
\newcommand{\Zv}{\mathcal{Z}}
\newcommand{\Zvunc}{\mathcal{Z}_{\text{unc}}}
\newcommand{\old}[1]{}

\newcommand{\GL}{{\text GL}_2(\C)}
\newcommand{\Zunc}{Z_{\mathrm{unc}}}
\newcommand{\Punc}{\mathrm{Pr}({\mathrm{unc}})}
\renewcommand{\Pr}{\mathrm{Pr}}
\newcommand{\p}{\mathcal{P}}
\newcommand{\Qdet}{\mathrm{qdet}}
\newcommand{\Tr}{\mathrm{Tr}}
\newcommand{\disk}[3]{\raisebox{-14bp}{\mbox{\hspace{2pt}
}\phantom{\rule{34bp}{37bp}}\hspace{2pt}}}

\begin{document}
\title{The Laplacian on planar graphs and graphs on surfaces}
\author{R. Kenyon\thanks{Brown University, 151 Thayer St., Providence, RI 02912. Research supported by the NSF.}}
\date{}
\maketitle

\tableofcontents
\section{Introduction}

By a {\bf resistor network}, or simply {\bf network}, we mean here
a graph $\G=(V,E)$ with vertices $V$
and (undirected) edges $E$, and $c$ a positive real-valued
function on the edges. The value $c_e$ is the {\bf conductance} of the edge $e$. 

The {\bf Laplacian} on a network $(\G,c)$ is the linear operator
$\Delta\colon\R^V\to\R^V$ defined by
\be\label{Lapdef}
(\Delta f)(v) =  \sum_{v'\sim v}c_{vv'}(f(v)-f(v')).\ee
Here the sum is over neighbors $v'$ of $v$.

The Laplacian is the most basic operator on a network, 
and has countless uses in all areas of mathematics.
We study here some specific properties of the
Laplacian on networks embedded on surfaces,
and in particular on the simplest surfaces: the plane, the annulus, the torus.

Our goals will be to focus on a few topological, probabilistic and combinatorial applications. Specifically, our three main interconnected goals are to discuss: 
\begin{enumerate}
\item The discrete EIT (Electrical Impedance Tomography) problem: reconstructing the network from boundary measurements. For planar networks, the 
classification due to Yves Colin de Verdi\`ere of Dirichlet-to-Neumann operators
\cite{CdV,CGV},
\item
The connections with the random spanning tree model developed by Curtis, Ingerman and Morrow \cite{CIM}
and Wilson and myself \cite{KW1, KW2}, 
and 
\item 
The characteristic polynomial of $\Delta$ on an annulus and torus
based on joint work with Okounkov and Sheffield \cite{KOS}, Okounkov \cite{KO}, and in \cite{Kenyon.bundle}.
\end{enumerate}
\bigskip

\noindent{\bf Acknowledgements.}
Large parts of section \ref{CPN} were developed by Colin de Verdi\`ere \cite{CdV} and Curtis, Ingerman and Morrow
\cite{CIM}. I thanks them for explanations and discussions.

The rest of this paper includes material developed in various projects jointly with Alexander Goncharov, Andrei Okounkov, Scott Sheffield, and David Wilson. In particular some
of the material is lifted directly out of work in progress with David Wilson \cite{KW4},
whom I thank for allowing it to be included here.

The main new results in the current paper are Lemma \ref{sequence}
and Theorems \ref{minimizable}, \ref{distinct}, and \ref{topequivtorus}.

\section{Background}
\subsection{Electrical Impedance Tomography}
The classical {\bf Calder\'on problem} is to determine the electrical conductivity
of a region in $\R^2$ or $\R^3$ 
by making boundary measurements of the type:
fix the voltage along the boundary and measure the resulting current flow
out of the region. 

This problem has concrete applications in medical imaging,
nondestructive materials testing, etc. In these settings it is called
Electrical Impedance Tomography (EIT). It is widely used because it is easy and cheap to implement:
basically it is very easy to apply a low voltage
to pretty much any conductive material and measure the resulting current flow.

In this paper we study the \emph{discrete version} of EIT. Given a 
resistor network (a finite graph in which every edge is a resistor),
understand its internal structure from boundary measurements.
There is a very nice solution for planar networks. For networks
on more general surfaces we are beginning to have a better understanding
but much more remains to be done.

For a network $\G$ on a surface with nontrivial topology there is closely related question,
even when there is no boundary, which is to understand the determinant of the
``bundle Laplacian" as a function on the moduli space of flat connections
on bundles on $\G$. 

We shall see how these questions are related below. 

\subsection{Spanning trees}
One of the combinatorial tools involved in the study of discrete EIT 
is the model of (random)  {\bf spanning
trees}, and related objects called {\bf cycle-rooted spanning forests}
on networks.
Kirchhoff \cite{Kirchhoff} was the first to see
the connection between spanning trees and the Laplacian on a network:
he showed that the Laplacian determinant in fact counts spanning trees.
 
The random spanning tree measure (or UST, for Uniform Spanning Tree)
has for the past $20$ years been a remarkably successful 
and rich area of study in probability theory.
The random spanning tree as a probability model goes back to at least
work of Temperley \cite{Temperley}, who gave a bijection between
spanning trees on the square grid and domino tilings.
Aldous and Broder \cite{Al,Br} in the 1980s showed a
connection between spanning trees and random walk,
giving an algorithm for sampling a UST based on 
the simple random walk on the network.
Pemantle \cite{Pemantle}
showed that the unique path between two points in 
a UST has the same distribution as the 
{\bf loop-erased
random walk} (LERW) between those two points\footnote{The
loop-erased random walk 
between $a$ and $b$ is defined as follows:
talk a simple random walk from $a$ stopped when it reaches $b$,
and then erase from the trace all loops in chronological order.
What remains is a simple path from $a$ to $b$.}. 
Wilson \cite{Wilson}
extended this to give a quick and elegant method of sampling
a uniform spanning tree in any graph.
Burton and Pemantle \cite{BP}
proved that the edges of the uniform spanning tree form 
a ``determinantal process'', that is there is a matrix $T$
indexed by the edges such that the probability that some set of edges
$\{e_1,\dots,e_k\}$ is in the UST is $\det(T(e_i,e_j)_{1\le i,j\le k})$. 
The matrix $T$ is built from the Green's function of $\Delta$.

The UST on the grid $\Z^2$ received particular attention due to 
the conformal invariance properties of its scaling limit
(limit when the mesh size of the grid goes to zero). 
In particular in \cite{K.ci} we showed that the ``winding field"
of the UST on $\Z^2$ (which measure the winding of the branches around faces) 
converges in the scaling limit
to the Gaussian free field, a conformally invariant random function
on $\R^2$.
In \cite{LSW} Lawler, Schramm and Werner showed the convergence of the
branches of the UST on $\Z^2$ in the scaling limit
to the conformally invariant process $\text{SLE}_2$.

As mentioned above, Temperley \cite{Temperley} 
gave a bijection between spanning trees and the dimer
model on the square grid; this was later generalized to arbitrary
graphs in \cite{KPW} and \cite{KS}, pointing to a very close relation
between the UST and the dimer model. Many of the results discussed
below are inspired by the corresponding results in the dimer model.

\subsection{Characteristic polynomials}

To a network on an annulus we can associate a polynomial $P(z)$ 
which is the determinant of the line-bundle Laplacian $\Delta$ 
for a flat line bundle with monodromy $z\in\C^*$ around a generator
of the homology of the annulus. 
One can think of this as the determinant of the Laplacian acting 
on the space of $z$-quasiperiodic functions: ``locally defined" functions whose values are multiplied by $z$
under analytic continuation counterclockwise around the annulus. 

The polynomial $P$ is real and symmetric 
(that is, $P(z)=P(1/z)$).
There is a nice characterization of those $P$ which occur:
they are those with positive constant coefficient, and whose roots are 
real, positive and distinct except for a double root at $z=1$. 
See Theorem \ref{distinct} below.

For a network on the torus one can associate a 
two-variable polynomial $P(z_1,z_2)$ in a similar way:
it is the determinant of the Laplacian for a flat line bundle with monodromy
$z_1,z_2$ around generators of $\pi_1$ of the torus.
This polynomial plays
an important role in understanding the structure and large-scale
properties of the random spanning tree model on a periodic planar graph
\cite{BP, KOS}.
Like in the case of the annulus there is a characterization of those $P(z_1,z_2)$
which occur: 
the curve $\{P(z_1,z_2)=0\}$ is a \emph{Harnack
curve} which is symmetric $(P(z_1,z_2)=P(z_1^{-1},z_2^{-1})$),
satisfies $P(1,1)=0$ and has positive constant coefficient. 
See the definition of Harnack curve in section \ref{Harnackdef}.
One can show that in fact any Harnack curve satisfying the above properties
arises from the characteristic polynomial
of some toral network. 

A more recent development in networks on the torus 
is the understanding of the integrable
dynamical system underlying them. 
Here the dynamical system is not related to the flow of current
in the network, but is a dynamical system on the space of conductances
assigned to the edges in the network. It is the restriction of an
integrable Hamiltonian
dynamical system on a larger space (the space of line bundles on
bipartite networks on the torus)
to a Lagrangian subvariety. See \cite{GK}. We will not discuss this here,
mostly because we don't yet have a very good understanding
of this system except as an invariant subvariety of a larger system.

For networks on other surfaces $\Sigma$
one can also define a characteristic polynomial $P=\det\Delta$
which is a regular function on the representation variety of $\pi_1(\Sigma)$
into $\SL$. Here much less is known about the structure of $P$. See 
however Theorem \ref{generalCRSF} below.

Finally, one would like to consider a network on a general surface with boundary,
and understand properties of the boundary measurement map
(for the appropriate bundle Laplacian) and to what extent it determines the network. We are still far from this goal but the cases of simple
topology or small boundary hint that there
is an interesting answer.

\section{Networks}

By a {\bf network} we mean a connected 
finite graph $\G=(V,E)$ with a positive real-valued function $c$ on the edges called the {\bf conductance}
function, and a usually nonempty set of vertices $\No$ called 
{\bf nodes} or {\bf terminals} or {\bf boundary vertices}.
These are the simplest type of electrical networks, with resistors only.

The nodes are indexed from $1$ to $n$.
Let $I$ be the set of non-node vertices; elements of $I$ are
referred to as {\bf internal vertices}.

\subsection{Laplacian}

The Laplacian on a network $\G$ is the operator $\Delta:\R^V\to\R^V$ defined for
$f\in\R^V$ by (\ref{Lapdef}). 
Since $\R^V$ has a natural basis, consisting of functions $\delta_v$
which are $1$ at $v$ and zero elsewhere, we often write $\Delta$ as a matrix
$\Delta=(\Delta_{vv'})_{v,v'\in V}$ where
$\Delta_{vv'}=-c_{vv'}$ if $v,v'$ are adjacent and the diagonal entries are
$\Delta_{vv}=\sum_{v'\sim v} c_{vv'}.$ See for example (\ref{Lapex})
which is the Laplacian of the graph of Figure \ref{Ynetwork}.

Associated to the Laplacian is the
quadratic form $Q$ on $\R^V$ defined as
\be\label{qf}
Q(f) = \langle f,\Delta f\rangle = \sum_{v\sim v'} c_{vv'}(f(v)-f(v'))^2
\ee
which is called the {\bf Dirchlet energy} of $f$. Here $\langle,\!\rangle$ is
the natural inner product in which the basis $\{\delta_v\}_{v\in V}$ is orthonormal.

From the expression (\ref{qf}) we see that $\Delta$ is positive semidefinite
with kernel consisting of the constant functions 
(we assume unless stated otherwise that $\G$ is connected).

Another useful operator is the {\bf incidence matrix} $d$, defined as follows.
Fix an (arbitrary) orientation for each edge $e$, so that $e_+$ and $e_-$
are its head and tail. 
Define $d\colon\R^V\to\R^E$  by the formula
$$df(e) = f(e_+)-f(e_-).$$
Then it is easy to verify that 
$$\Delta =d^*{\cal C}d,$$
where ${\cal C}$ is the diagonal matrix of conductances, and $d^*$ is the transpose
of $d$ (its adjoint for the inner product $\langle,\rangle$ above): 
for an element $w\in\R^E$ we have
$$d^*w(v) = \sum_{e\to v} w(e)$$
where the sum is over the edges $e$ containing endpoint $v$, and directed towards $v$. We can call $df$ the {\bf gradient} of the function $f$ and
$d^*w$ the {\bf divergence} of the flow $w$.

Note that we can write the Dirichlet energy as
$$Q(f)=\langle df,{\cal C}df\rangle$$
where here $\langle,\!\rangle$ is the natural inner product on $\R^E$.

\subsection{The Laplacian with boundary}

If $B\subset V$ is a subset of vertices
and $I=V\setminus B$, 
we can define an operator
$\Delta_B:\R^I\to\R^I$ called the {\bf Dirichlet Laplacian with boundary $B$}
by the same formula
$$\Delta_Bf(v) = \sum_{v'\sim v}c_{vv'}(f(v)-f(v')),$$
where here $v\in I$ and the sum is over all neighbors of $v$ in $V$
(not just those in $I$).

As a matrix $\Delta_B$ is just a submatrix of the full Laplacian $\Delta$:
the submatrix indexed by $I$.

\subsection{The Dirichlet problem}

A function $f\in\R^V$ is said to be {\bf harmonic at $v$} if $\Delta f(v)=0$.
By (\ref{Lapdef}) this simply means that $f(v)$ is a weighted average
of the values of $f$ at the neighboring vertices, where the weights
are proportional to the conductances. 

A function is {\bf harmonic} if it is harmonic at all vertices.
On a finite connected graph a harmonic function must be constant,
since the kernel of $\Delta$ consists of constant functions only.

Suppose $B$ is a nonempty set boundary vertices and $u$ is a function on $B$.
A function $f$ on $V$ which is harmonic on $I=V\setminus B$ and agrees with
$u$ on $B$ is said to be {\bf harmonic 
with Dirichlet boundary conditions on $B$}. Such functions satisfy
a {\bf maximum principle}: their maximum (and minimum) values occur on $B$
(because of the above observation about weighted
averages).

Given this fact it is not hard to see that, if $B$ is nonempty, then $\Delta_B$ is invertible:
any function $f$ in the kernel of $\Delta_B$ can be extended to a function on $V$ taking values zero on $B$; it is thus a harmonic function with Dirichlet boundary 
conditions at $B$. By
the maximum principle it must take its maximum
and minimum on $B$, where it has value $0$. Thus $f$ must be identically zero.

The {\bf Dirichlet problem} is the problem of
finding, for a given function $u$ on $B$, a function $f$ on $V$ which is
equal to $u$ on $B$ and harmonic on $I=V\setminus B$: find $f$
satisfying
\be\label{DP}
\begin{aligned}
\Delta f(v) =0 &\text{ for }v\in V\setminus B\\
f(v) = u(v) &\text{ for } v\in B.
\end{aligned}
\ee
The function $f$ is called the {\bf harmonic extension} of $u$. 
For finite networks the solution to the Dirichlet problem is straightforward
using invertibility of $\Delta_B$. Alternatively,
the solution $f$ is the function with boundary values $u$ and
which minimizes the Dirichlet energy $Q(f)$. Convexity of $Q$ implies that
there is a unique solution. 

A function on $\G$ which is harmonic on $I$ is a solution to the Dirichlet
problem with boundary $B$. Such a function is also called
an {\bf equilibrium potential}.

\subsection{The Dirichlet problem and the response matrix}

For a network we will take $\No$, the set of nodes, to play the role of the 
boundary vertices (which we called $B$ above). 
Given a function $u$ on $\No$, let $f_u$ be the harmonic
extension of $u$. Then $\Delta f_u(v)$, restricted to $\No$, is a (generally nonzero) function of $v\in\No$. 
The map $\Lambda:u\mapsto \Delta f_u|_\No$  is a linear 
map $\Lambda:\R^\No\to\R^\No$ called
the {\bf Dirichlet-to-Neumann map} of the network.
For certain conventional reasons we often prefer to work with $L=-\Lambda$
and refer to $L$ as the {\bf response matrix}. See the next section for an example.

We can describe $\Lambda$ as a matrix as follows.
Suppose we index the vertices so that vertices in $\No$ come first.
Then we can write $\Delta$ in block form as
$$\Delta = \left(\begin{matrix}A&B\\B^t&C\end{matrix}\right),$$
where $A$ has rows and columns indexed by the nodes $\No$
and $C$ has rows
and columns indexed by the internal vertices; $C$ is just the matrix
of the Dirichlet Laplacian $\Delta_\No$. 

The Dirichlet problem (\ref{DP}) can be written as follows: find $f_I$ on 
the internal vertices $I$ so that 
$$\left(\begin{matrix}A&B\\B^t&C\end{matrix}\right)
\left(\begin{matrix}u\\f_I\end{matrix}\right)=\left(\begin{matrix}g\\0\end{matrix}\right),$$ for a function $g$ which is the desired function of $u$.
We can split this into two linear equations
\begin{eqnarray}
Au+Bf_I &=& g\\B^t u+C f_I &=&0.
\end{eqnarray} 
This second equation can be solved as 
\be\label{f_I}f_I=-C^{-1}B^t u.\ee
(Here $C$ is invertible, as discussed above).

One can then plug in this value of $f_I$ in the first equation to get 
$$Au-BC^{-1}B^t u = g,$$ that is,
$\Lambda u=g$ where $\Lambda$ is the matrix $\Lambda=A-BC^{-1}B^t$. 
Thus
\be\label{responsedef}L=-A+BC^{-1}B^t.\ee

We see from this expression that $L$ is symmetric, since $A$ and $C$
are symmetric.
In fact $\Lambda$ is positive semidefinite (and $L$ is negative semidefinite),
with kernel consisting solely of the constant functions. 
This follows from the fact that 
$$\langle u,\Lambda u\rangle = \langle f,\Delta f\rangle\ge 0$$
where $f$ is the harmonic extension of $u$ (one can check this using (\ref{f_I})). 

\subsubsection{Example}

Let $\G$  be the network of Figure \ref{Ynetwork} with three nodes 
$\{v_1,v_2,v_3\}$,
one internal vertex $v_4$
\begin{figure}[htbp]
\center{\includegraphics[width=6cm]{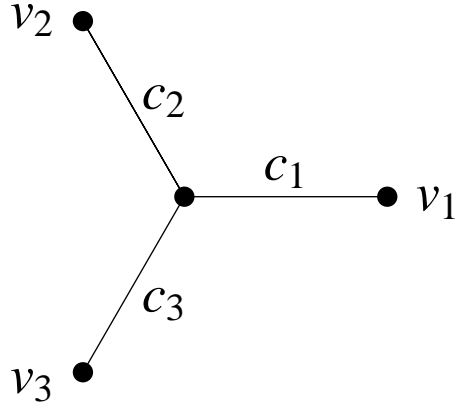}}
\caption{\label{Ynetwork}The `Y' network.}
\end{figure}
and three edges with conductances $c_1,c_2,c_3$.
We have
\be\label{Lapex}
\Delta = \begin{bmatrix}c_1&0&0&-c_1\\0&c_2&0&-c_2\\0&0&c_3&-c_3\\-c_1&-c_2&-c_3&c_1+c_2+c_3\end{bmatrix}.\ee
Using formula (\ref{responsedef}) we have
\begin{eqnarray}\nonumber
L&=&-\begin{pmatrix}c_1&0&0\\0&c_2&0\\0&0&c_3\end{pmatrix}+\begin{pmatrix}-c_1\\-c_2\\-c_3\end{pmatrix}(\frac1{c_1+c_2+c_3})\begin{pmatrix}-c_1&-c_2&-c_3\end{pmatrix}\\\label{Lexample}
&=&\begin{pmatrix}\frac{-c_1c_2-c_1c_3}{c_1+c_2+c_3}&\frac{c_1c_2}{c_1+c_2+c_3}&\frac{c_1c_3}{c_1+c_2+c_3}\\\frac{c_2c_1}{c_1+c_2+c_3}&\frac{-c_2c_1-c_2c_3}{c_1+c_2+c_3}&\frac{c_2c_3}{c_1+c_2+c_3}\\
\frac{c_3c_1}{c_1+c_2+c_3}&\frac{c_3c_2}{c_1+c_2+c_3}&
\frac{-c_3c_1-c_3c_2}{c_1+c_2+c_3}\end{pmatrix}.
\end{eqnarray}

\subsection{Electrical interpretation}

Associated to a function $f$ (called {\bf potential} or {\bf voltage}) 
on a network is a flow $\omega$ on the edges (called {\bf current})
defined by 
$$\omega(e)={\cal C}df(e)=c_e(f(e_+)-f(e_-)).$$
Here ${\cal C}$ is the diagonal matrix of edge conductances and 
$df(e)=f(e_+)-f(e_-)$ is the gradient of $f$.

Ohm's law states that ${\cal C}df(e)$ is the flow across edge $e$ when the vertices
are held at potentials $f(e_+)$ and $f(e_-)$, and the {\bf resistance} on the edge
is $1/c_e$, the reciprocal of the conductance.

Kirchhoff's network equations state that when the boundary 
nodes are held at fixed
potentials, the internal vertices will attain potentials $f$ in such a way
that the flow of current induced by Ohm's law in the network has the property that
the current entering an internal vertex equals the current exiting that vertex.
In other words the divergence $d^*\omega$ of the current
$\omega$ is zero at internal vertices.
Thus $d^*{\cal C}d(f)=0$ at every internal vertex, so $f$ is the harmonic 
extension of $u$.

This allows us to give $L_{ij}$ the following electrical interpretation:
Hold all nodes except $i$ at potential $0$, and node $i$ at potential $1$. Then
for the resulting current flow,
$L_{ij}$ is the current exiting the network at node $i$.
This current is positive: the potentials at internal vertices are in $[0,1]$
(in fact in $(0,1)$)
by the maximum principle for harmonic functions, 
so the current at edges adjacent to $j$ is directed towards $j$.
Similarly one shows $L_{ii}<0$.

\section{Circular planar networks}\label{CPN}

By a {\bf circular planar network} (CPN) we mean a connected 
finite graph $\G=(V,E)$ embedded in the plane,
with a positive real-valued function $c$ on the edges called the {\bf conductance}
function, and a nonempty set of vertices $\No$ on the outer face called 
{\bf nodes} or {\bf terminals}. 
(There may be other vertices on the outer face as well).
The reason it is a \emph{circular} planar network is that the network
can be drawn inside the disk with the nodes on the boundary circle. 

The nodes are ordered in counterclockwise order from $1$ to $n$.
Let $I$ be the set of non-node vertices; elements of $I$ are
referred to as {\bf internal vertices}.
The non-exterior faces of $\G$ are referred to as {\bf interior faces};
we consider $\G$ to also have $n$ {\bf exterior faces}, with the $i$th
exterior face comprising the region between nodes $i$ and $i+1$,
and having
boundary consisting of the edges along the ``outer face" of $\G$ between
node $i$ and $i+1$.  

\subsection{Conjugate harmonic function}

If $\G$ is a circular planar network,
a harmonic function $f$ on $\G$ with Dirichlet boundary conditions on $\No$
has a {\bf conjugate harmonic function} $g$
defined (up to an additive constant) on the faces of $\G$ by the condition
${\cal C}df = \partial g$, where $\partial$ is the difference
operator on the dual graph $\G'$: if $e=vv'$ is an edge of 
$\G$ and $F_1,F_2$ are the adjacent faces, so that $F_1$ is left of 
$e$ when $e$ is traversed from $v$ to $v'$, then 
\be\label{CR}c_e(f(v')-f(v))=g(F_1)-g(F_2).\ee
The harmonicity of $f$ implies that these equations are consistent
around a vertex; by summing along paths one sees that the value of $g$
is defined consistently on all faces, once one choose a value
of $g$ at any starting face.
Thus these equations define $g$ uniquely up to a global additive constant.
The function $g$ is harmonic for the operator $\partial^*{\cal C}^{-1}\partial$.
That is, one puts conductances on the dual edges which are the inverses
of the conductances on the primal edges.

The pair $(f,g)$ is sometimes written as $f+ig$ and is called a {\bf discrete 
analytic function} since the equations (\ref{CR}) play the role of 
the discrete Cauchy-Riemann equation.

If $f$ is harmonic with Dirichlet boundary conditions then $g$ will be
harmonic on $\G'$ with Dirichlet boundary conditions, that is,
$g$ is harmonic on interior vertices of $\G'$ (that is, on 
non-exterior faces of $\G$).

\subsection{Characterization of response matrices}

The response matrix is a fundamental matrix associated to a network.
A natural question is: what properties of the network are reflected in its response
matrix? 

The basic properties of $L$ shown in the previous section are that 
\begin{enumerate}\item $L$ is real and symmetric, 
\item $L$ is negative semidefinite, 
\item The kernel of $L$ consists of the constant functions, and
\item 
$L_{ij}>0$ if $i\ne j$, and $L_{ii}<0$. 
\end{enumerate} 
(It is not hard to see that property 2 is implied by the others).

Any matrix satisfying the above properties
is in fact the response matrix
of some (not necessarily planar) network: it suffices to take a complete graph
on $n$ vertices, declare that each vertex is a node, and put conductance
$L_{ij}$ on the edge $ij$ connecting vertices $i$ and $j$.

The question is much more interesting in the case of 
circular planar networks.
In this case there is a beautiful characterization of response matrices
in terms of minors, due to Colin de Verdi\`ere \cite{CdV}.

Given two disjoint subsets $A,B$ of the nodes of a circular planar network, 
we say that $A$ and $B$
are {\bf noninterlaced} if they are contained in disjoint intervals
of nodes with respect to the circular order. That is, if $a<b<c<d$ are nodes
then it is never the case that $a,c\in A$ and $b,d\in B$. 

\begin{thm}[\cite{CdV}]\label{CdVthm}
$L$ is the response matrix of a connected circular planar network if and only if
$L$ is symmetric, with kernel consisting of the constant functions,
and for any disjoint noninterlaced subsets of nodes $A$ and $B$ with $|A|=|B|$, we have
$\det L_A^B\ge 0$, where $L_A^B$ is the submatrix of $L$ whose
rows index $A$ and columns index $B$.
\end{thm}

There are two parts of this theorem, the necessity of the conditions
and the existence of a network for any $L$ satisfying the conditions. 
Both have interesting proofs which we will discuss. 
\old{Even though there
are an exponentially large set of inequalities in the condition,
in recent work \cite{KW4} we showed that $\begin{pmatrix}n\\2\end{pmatrix}$
of these inequalities imply the others. See Theorem \ref{centralminorthm} below.
}

We define a network to be {\bf well-connected} if all noninterlaced 
minors $\det L_A^B$ are \emph{strictly} positive.
In \cite{CdV} Colin de Verdi\`ere shows that the space of well-connected
circular planar networks on $n$ nodes is homeomorphic to an open ball
of dimension $\begin{pmatrix}n\\2\end{pmatrix}$.
More about this in section \ref{YDeltamoves}

\subsection{Groves}
As one can see from (\ref{responsedef}), the response matrix entries are rational functions
of the conductances. Surprisingly, the off-diagonal entries are {\bf subtraction free} rational functions
of the conductances, that is, ratios of polynomials with positive coefficients;
in fact the coefficients are all $1$ or $0$ (see for example (\ref{Lexample})).
This is a strong indication that these polynomials are counting some
type of combinatorial objects.

Indeed, that is one of the ideas in the Curtis-Ingerman-Morrow proof 
(see  \cite{CIM}) of 
necessity in Theorem \ref{CdVthm} above, which we discuss in this section.

\subsubsection{Spanning trees}
A {\bf spanning tree} of a graph is a subset of edges which is a tree,
that is, contains no cycles, and is spanning, that is, connects all vertices.
Equivalently one can say that a spanning tree is a maximal set of edges
(with respect to inclusion) that has no cycles.

For a network $\G$ with conductances $c$ on its edges, one can define
the {\bf weight} $w(T)$ of a spanning tree $T$ to be the product of the 
conductances of its edges, $w(T)=\prod_{e\in T}c_e$.

The most important theorem about spanning trees is 
that the Laplacian determinant is the weighted sum of spanning trees
(this is also called the {\bf partition sum}).
\begin{thm}[Matrix-Tree Theorem \cite{Kirchhoff}]\label{matrixtree}
Let $\det_0\Delta$ be the product of the nonzero eigenvalues (with multiplicity)
of $\Delta$ for a connected network $\G$. Then
$$\frac1{|\G|}{\det}_0\Delta = \sum_Tw(T).$$
Here the sum is over spanning trees of $\G$.
\end{thm}

This theorem is usually attributed to Kirchhoff \cite{Kirchhoff}.
An equivalent formulation is that the partition sum $Z:=\sum_Tw(T)$
is equal to 
the {\bf reduced determinant} of $\Delta$,
which is the determinant of $\Delta$ after removing any row and column;
since $\Delta$ is symmetric and constant functions are in the kernel
any choice of row and column gives the same reduced determinant.
\medskip

\noindent{\emph{Proof of Theorem \ref{matrixtree}, sketch.}}
Several proofs of the Matrix-tree Theorem are known; the classical
one we present uses the {\bf Cauchy-Binet Theorem} 
which is a formula for computing
the determinant of a product of two nonsquare matrices: it states that
if $X$ is an $n\times m$ matrix and $Y$ is $m\times n$ with $m\ge n$ then
\be\label{CB}
\det XY=\sum_B \det(X^B)\det(Y_B)\ee 
where $B$ runs over all choices of $n$ column
indices for $X$, $X^B$ is the $n\times n$ minor of $X$ defined by columns $B$
and $Y_B$ is the $n\times n$ minor of $Y$ defined by rows $B$. 

For the Matrix-Tree Theorem, let $\Delta_{\bar 1,\bar 1}$ be the submatrix
of $\Delta$ obtained by removing row $1$ and column $1$.
Then $\Delta_{\bar 1,\bar 1}=d^*_{\bar 1}{\cal C}d_{\bar 1}$ where $d_{\bar 1}$
is obtained from $d$ by removing column $1$.

One now uses $X=d_{\bar 1}^*$ and $Y={\cal C}d_{\bar 1}$ in (\ref{CB}). 
Each $B$ in (\ref{CB}) is a set of $|V|-1$ edges; we claim that 
the nonzero summands in 
(\ref{CB}) are precisely the sets of edges which form spanning trees. If $B$ forms a tree, its edges can be directed
towards the removed vertex. Then in the 
definition of $\det d_{\bar 1,B}$ as an expansion over the symmetric group,
$$\det d_{\bar 1,B}=
\sum_{\sigma\in S_{n-1}}(-1)^\sigma d_{2\sigma(2)}\dots d_{n\sigma(n)},$$
the only nonzero term is the one whose permutation $\sigma$
matches each vertex to
the edge adjacent to it and closer to the root. 
Thus $\det d_{\bar 1,B}=\pm1$. If edges in $B$ do not form a tree, 
then there must be
more than one component; the function on vertices which is $1$ on a
component not containing the removed vertex and zero elsewhere is in the
kernel of $d_{\bar 1,B}$, so $\det d_{\bar 1,B}=0$. 

Finally,
the determinant of the diagonal submatrix $\det {\cal C}_B^B$ for a tree $B$
is the product of the edge weights, and 
$\det d_{\bar 1,B}^*=\det d_{\bar 1,B}=\pm1$.
So $$\det(d^*_{\bar 1,B}{\cal C_B^B}d_{\bar 1,B})=w(T).$$ This completes the proof.
\hfill$\square$
\medskip

The {\bf spanning tree measure} is the probability measure on spanning
trees of a graph in which the probability of a tree is proportional to its weight.
In the special case that the conductances are all $1$ this is the {\bf uniform
measure on spanning trees} (UST measure); 
each tree has probability $\frac1{\kappa}$ where
$\kappa$ is the determinant of the reduced Laplacian of the conductance-$1$
network; $\kappa$ is sometimes called the {\bf complexity} of the network.

Cayley's formula \cite{Cayley}, first proved by Borchhardt \cite{Borch}
gives the number spanning trees of  the complete
graph $K_n$ to be $n^{n-2}$. This follows easily from the above theorem
since $$\Delta=\begin{pmatrix}n-1&-1&\dots&-1\\
-1&n-1&-1&\vdots\\
\vdots&&\ddots&-1\\
-1&\dots&-1&n-1\end{pmatrix}$$
whose eigenvalues are $n$ with multiplicity $n-1$ and $0$ with multiplicity $1$.

\subsubsection{Groves from minors of $L$}
On a network with nodes $\No$, a {\bf grove} is a subset of the edges,
each of whose components is a tree containing one or more nodes.
In other words it is a ``spanning forest"  in which each component is
connected to the boundary (and isolated vertices are
considered components, so there are no isolated internal vertices). 
See Figure \ref{groveexample}.
We allow a component to contain more
than one node; for example a spanning tree is a special case of a grove. 

\begin{figure}
\center{\includegraphics[width=2in]{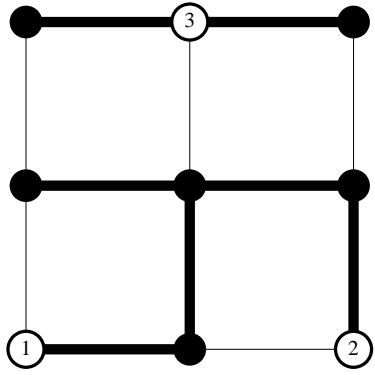}}
\caption{\label{groveexample}A grove of type $12|3$ (nodes are in white;
internal vertices in black).}
\end{figure}
Groves can be put into equivalence classes depending on how their 
components partition the nodes. Given a partition $\pi$ of the nodes,
let $S_\pi$ be the set of groves in which each component contains exactly all the nodes
of a part of $\pi$. For example when there are three nodes the
partitions are: $1|2|3$ where each node is in a separate component, 
$1|23$ where nodes $2$ and $3$ are in the same component different
from $1$, and likewise $2|13, 3|12,$ and finally $123$ where all nodes
are in the same component.  
More generally groves in the set $S_{1|2|\dots|n}$, that is, in
which each component contains exactly one node, are called
{\bf uncrossings}. The groves in $S_{12\dots n}$ are spanning trees. 

Note that for circular planar graphs not all partitions $\pi$
are possible; for a partition to be obtained by a grove
it must be planar. For example $13|24$ is not the partition of any grove on
a circular planar graph. 

The probability measure on spanning trees extends to a probability measure
on groves: the weight of a grove is the product of its edge weights,
and the probability of a grove is $\frac1{Z}$ times its weight,
where $Z=\sum_T w(T)$ is the weighted sum of all groves. 

For each partition $\pi$ we denote by $Z_{\pi}$ the weighted sum of groves
of type $\pi$. Then $Z=\sum_{\pi}Z_{\pi}$, and the probability that
a random grove has type $\pi$ is $\frac{Z_{\pi}}{Z}$. 
We define $\Zunc=Z_{1|2|\dots|n}$ to be the partition sum for uncrossings,
and $Z_{\text{tree}}=Z_{12\dots n}$ to be the weight sum of spanning trees.

All minors of the response matrix have combinatorial interpretations in terms of groves, for example when $i\ne j$ we prove in Theorem \ref{LABC} below that
$$L_{ij}=\frac{Z_{\pi(ij|\text{rest different})}}{\Zunc}=\frac{\Pr(\pi(ij|\text{rest different}))}{\Punc}$$ 
where $\pi(ij|\text{rest different})$ is the partition
in which all nodes are in separate components except for $i$ and $j$ which
are in the same component. 
More generally, we have the following combinatorial interpretation of \emph{all}
minors of $L$ (a closely related result
can be found in \cite{CIM}):
\begin{theorem}[\cite{KW2}]\label{LABC}
Let $Q,R,S,T\subset\No$ be a partition of $\No$ into four sets with $|R|=|S|$
(and some of which may be empty).
  Then $\det L_{R\cup T}^{S\cup T}$ is the ratio of two terms: the denominator
  is $\Zunc$; the numerator is a signed weighted sum
  of groves of $\G_T$, the graph $\G$ in which all nodes in $T$ are
  considered internal vertices, the nodes in $Q$ are in singleton parts,
  and in which nodes in $R$ are paired with nodes
  in $S$, with the sign being the sign of the pairing permutation:
\be\label{generalminor}
 \det L_{R\cup T}^{S\cup T} = (-1)^{|T|}\sum_{\text{\textrm{permutations} $\rho$}}(-1)^\rho\frac{Z\big[{}_{r_1}^{s_{\rho(1)}}|\cdots|_{r_k}^{s_{\rho(k)}}|q_1|\cdots|q_\ell\big]}{\Zunc}.
\ee
\end{theorem}

As a special case of this theorem, we have the above statement about $L_{ij}$
for $i\ne j$ (here $L_{ij}$ is the $i,j$-entry of $L$ which in the notation of
the theorem would be $L_{\{i\}}^{\{j\}}$). 
Another special case is the reduced determinant of $L$: for any $i$ and $j$,
removing row $i$ and column $j$ gives
$$\det L_{\bar i}^{\bar j}=\frac{Z_{\text{tree}}}{\Zunc}$$
(to see this, take $R=\{i\},S=\{j\}$ and $T=\No\setminus\{i,j\}$).

In the example (\ref{Lexample}) above where $\G$ is the $Y$ network 
of Figure \ref{Ynetwork},
$L_{12}=\frac{c_1c_2}{c_1+c_2+c_3}$; $c_1c_2$ 
is the weight of the unique grove of type
$12|3$ and $c_1+c_2+c_3$ is the sum of the weights of the three
groves of type $1|2|3$.
\medskip

\noindent{\it Proof of Theorem \ref{CdVthm}, necessity}.
Since nonplanar pairings do not occur,
if $R$ and $S$ are non-interlaced subsets of the nodes of the same size,
then 
$\det L_R^S$ is a nonnegative because there is a single term in 
the sum in (\ref{generalminor}) and the permutation is the identity.
This gives a combinatorial interpretation of the inequalities in 
Theorem \ref{CdVthm}: they are positive because they count something. 
\hfill{$\square$}

\begin{proof}[Proof of Theorem \ref{LABC}]
Write $\Delta$ in block form as 
$$\Delta=\begin{pmatrix}A&B\\B^t&C\end{pmatrix}$$
where $A$ is indexed by the nodes and $C$ by the internal vertices.
Then we can factor $\Delta$ as
$$\Delta=\begin{pmatrix}A&B\\B^t&C\end{pmatrix}=
\begin{pmatrix}A-BC^{-1}B^t&BC^{-1}\\0&I\end{pmatrix}
\begin{pmatrix}I&0\\B^t&C\end{pmatrix}=\begin{pmatrix}-L&BC^{-1}\\0&I\end{pmatrix}
\begin{pmatrix}I&0\\B^t&C\end{pmatrix}.
$$

We wish to compute the determinant of the minor
$\Delta_{X\cup I}^{Y\cup I}$. For notational simplicity reorder the nodes so that 
nodes in $X$ come last and nodes in $Y$ come first.
Then $$\Delta=\begin{pmatrix}*&*&*\\-L_X^Y&*&*\\0&0&I\end{pmatrix}
\begin{pmatrix}I&0&0\\0&I&0\\{}*&*&C\end{pmatrix}.
$$
and
$$\Delta_{X\cup I}^{Y\cup I}=\begin{pmatrix}-L_X^Y&*&*\\0&0&I\end{pmatrix}
\begin{pmatrix}I&0\\0&0\\{}*&C\end{pmatrix}.$$

The Cauchy-Binet formula (\ref{CB}) then gives
$$\det\Delta_{X\cup I}^{Y\cup I}=(-1)^{|X|}\det L_X^Y\det C.$$ 
In particular
$$\det\Delta_{R\cup T\cup I}^{S\cup T\cup I}=(-1)^{|R|+|T|}
\det L_{R\cup T}^{S\cup T}\det C
$$ 
where $I$ is the set of internal nodes. 
Since $\Zunc=\det C$, it suffices to show that 
$\det\Delta_{R\cup T\cup I}^{S\cup T\cup I}$ is the desired numerator.

Let $I'=I\cup T$ be the ``new" set of internal nodes.
We need to evaluate $\det\Delta_{R\cup I'}^{S\cup I'}$.
The remainder of the proof is now an extension of 
the proof of Theorem \ref{matrixtree}.
Write $\Delta=d^*{\cal C}d$.
Then $\Delta_{R\cup I'}^{S\cup I'}=d_{S\cup I'}^*{\cal C}d_{R\cup I'}$
where $d_X$ represents the restriction of $d$ to the subspace indexed
by $X$.
By the Cauchy-Binet theorem,
\be\label{multisum}
\det d_{S\cup I'}^*{\cal C}d_{R\cup I'}=\sum_Y\det (d^Y_{S\cup I'})^*
\det{\cal C}_Y^Y
\det d_{R\cup I'}^Y,\ee
where the sum is over collections of edges $Y$ of cardinality $|S\cup
I'|=|R\cup I'|$.  

The terms in the sum (\ref{multisum}) are collections of edges $Y$ in the graph 
$\G$. The number of components of the subgraph with edges $Y$ is at least 
$$|\text{vertices}|-|\text{edges}|=|R\cup S\cup Q\cup I'|-|Y|=
|R|+|Q|.$$ We claim that for each nonzero term in the sum (\ref{multisum}),
each component is a tree containing exactly one element of $S\cup Q$.
If some component did not contain a point of $S\cup Q$, then the function
which is $1$ on this component and zero elsewhere would be in the kernel
of $d_{R\cup I'}^Y$ whose determinant would then be zero. 
Since the number of components is $|R|+|Q|=|S|+|Q|$ each component must contain exactly one point of $S\cup Q$.
Similarly each component also contains exactly
one element of $R\cup Q$ (since $\det d_{S\cup I}^Y\ne 0$).
Thus each component is a tree containing either a unique $q\in Q$
(and no point of $R\cup S$), or a tree containing both a 
unique $r\in R$ and a unique $s\in S$. The total number of components
is $|R|+|Q|$ so each element of $Q,R$ and $S$ is in a component.
In particular there is a permutation $\rho$ and for each $i$ a 
component joining $r_i$, the $i$th element of $R$, to $s_{\rho(i)}$, the $\rho(i)$th
element of $S$.  The weight of a component
is the product of its edge weights.
The sign of the term corresponding to the subset $Y$
can be determined as follows. Each component connected to $Q$ occurs
in both $\det d_{R\cup I'}^Y$ and $\det d_{S\cup I'}^Y$ and so contributes
sign $+1$. If we postmultiply $d_{S\cup I'}$ by the permutation matrix $P$
which permutes the rows according to $\rho$ then the sign of the 
corresponding terms in $Pd_{S\cup I'}$ and $d_{R\cup I'}$ 
are the same since they corresponding to the same choice of rows.
Since $\det P= (-1)^{\rho}$ this completes the proof. 
\end{proof}

\subsubsection{Other groves}

Theorem \ref{LABC} shows that, when divided by the factor $\Punc$
(the probability of the uncrossing) 
certain grove probabilities are given by
noninterlaced minors of $L$: those groves which correspond
to noninterlaced pairings, for example.

Remarkably, one can give a similar expression for \emph{all} grove types:
see Theorem \ref{allgroves} below.

For a partition $\tau$ of $\{1,\dots,n\}$ (planar or not) we define 
\begin{equation}\label{Ltau}
L_\tau = \sum_F\prod_{\text{$\{i,j\} \in F$}} L_{i,j},
\end{equation}
where the sum is over those spanning forests $F$ of the complete graph on
$n$ vertices $1,\dots,n$ for which the trees of $F$ span the parts of $\tau$.
For example we have 
$$L_{(123|4)}=L_{12}L_{13}+L_{12}L_{23}+L_{13}L_{23}$$
and $$L_{(13|24)}=L_{13}L_{24}.$$

\begin{thm}[\cite{KW1}]\label{allgroves}
For any planar partition $\sigma$ we have
$$\frac{\Pr(\sigma)}{\Punc} = \sum_\tau \p_{\sigma,\tau} L_\tau$$
where $\p$ is the integer matrix defined below.
\end{thm}

The rows of the matrix $\p$ are indexed by planar partitions, and the columns are indexed by all partitions.  In the case of $n=4$ nodes, the matrix $\p$ is

\vspace{3pt}
\newcommand{\rf}[1]{\begin{rotatebox}{90}{$#1$}\end{rotatebox}}
\centerline{
\begin{tabular}{r@{$\,\,\,$}c@{$\,$}c@{$\,$}c@{$\,$}c@{$\,$}c@{$\,$}c@{$\,$}c@{$\,$}c@{$\,$}c@{$\,$}c@{$\,$}c@{$\,$}c@{$\,$}c@{$\,$}c@{$\,$}c}
        & \rf{1|2|3|4} & \rf{12|3|4} & \rf{13|2|4} & \rf{14|2|3} & \rf{23|1|4} & \rf{24|1|3} & \rf{34|1|2} & \rf{12|34} & \rf{14|23} & \rf{1|234} & \rf{2|134} & \rf{3|124} & \rf{4|123} & \rf{1234\phantom{|}} & \rf{13|24} \\
$1|2|3|4$  &1&0&0&0&0&0&0&0&0&0&0&0&0&0& 0 \\
$12|3|4$   &0&1&0&0&0&0&0&0&0&0&0&0&0&0& 0 \\
$13|2|4$   &0&0&1&0&0&0&0&0&0&0&0&0&0&0& 0 \\
$14|2|3$   &0&0&0&1&0&0&0&0&0&0&0&0&0&0& 0 \\
$23|1|4$   &0&0&0&0&1&0&0&0&0&0&0&0&0&0& 0 \\
$24|1|3$   &0&0&0&0&0&1&0&0&0&0&0&0&0&0& 0 \\
$34|1|2$   &0&0&0&0&0&0&1&0&0&0&0&0&0&0& 0 \\
$12|34$    &0&0&0&0&0&0&0&1&0&0&0&0&0&0&$-1$ \\
$14|23$    &0&0&0&0&0&0&0&0&1&0&0&0&0&0&$-1$ \\
$1|234$    &0&0&0&0&0&0&0&0&0&1&0&0&0&0& 1 \\
$2|134$    &0&0&0&0&0&0&0&0&0&0&1&0&0&0& 1 \\
$3|124$    &0&0&0&0&0&0&0&0&0&0&0&1&0&0& 1 \\
$4|123$    &0&0&0&0&0&0&0&0&0&0&0&0&1&0& 1 \\
$1234$     &0&0&0&0&0&0&0&0&0&0&0&0&0&1& 0.
\end{tabular}}
\vspace{3pt}
For example, the row for $1|234$ tells us
$$
  \frac{\Pr(1|234)}{\Pr(1|2|3|4)} = L_{(1|234)} + L_{(13|24)} = L_{2,3} L_{3,4} + L_{2,3}L_{2,4} + L_{2,4} L_{3,4} + L_{1,3}L_{2,4}
$$
and the row for $12|34$ tells us
$$
  \frac{\Pr(12|34)}{\Pr(1|2|3|4)} = L_{(12|34)} - L_{(13|24)} = L_{1,2} L_{3,4} - L_{1,3}L_{2,4}.
$$

We call this matrix $\p$ the \textbf{projection matrix from partitions to
planar partitions}, since it can be interpreted as a map from the
vector space whose basis vectors are indexed by all partitions to the
vector space whose basis vectors are indexed by planar partitions, and
the map is the identity on planar partitions.  
For example, the column for $13|24$ tells us
$$
13|24\ \ \text{projects to}\ \ -12|34-14|23+1|234+2|134+3|124+4|123.
$$

The projection matrix may be computed using some simple combinatorial
transformations of partitions.  Given a partition $\tau$, the
$\tau$th column of $\p$ may be computed by repeated application of
the following transformation rule, until the resulting formal linear
combination of partitions only involves planar partitions.
The rule is a generalization of the transformation
$$
13|24 \to -12|34-14|23+1|234+2|134+3|124+4|123
$$ 
which can be represented as a sort of skein diagram:
\begin{center}
\includegraphics[scale=1]{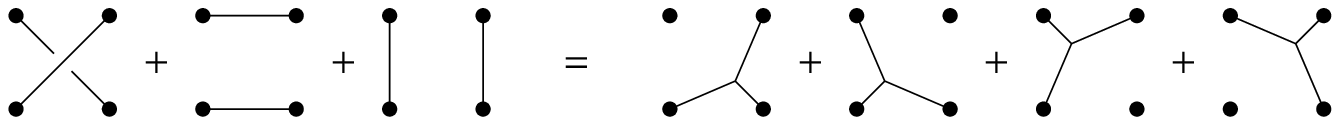}
\end{center}
We generalize this rule to partitions $\tau$ containing additional items and parts
as follows.  If
partition $\tau$ is nonplanar, then there will exist items $a<b<c<d$
such that $a$ and $c$ belong to one part, and $b$ and $d$ belong to
another part.  Arbitrarily subdivide the part containing $a$ and $c$
into two sets $A$ and $C$ such that $a\in A$ and $c\in C$, and
similarly subdivide the part containing $b$ and $d$ into $B\ni b$ and
$D\ni d$.  Let the remaining parts of partition $\tau$ (if any) be
\newcommand{\rest}{\operatorname{rest}}
denoted by ``$\rest$.''  Then the transformation rule is
\begin{multline}\label{part-trans}\tag{Rule~1}
 AC|BD|\rest \to
 A|BCD|\rest
\,+\,B|ACD|\rest
\,+\,C|ABD|\rest
\,+\,D|ABC|\rest\\
-\,AB|CD|\rest
\,-\,AD|BC|\rest.
\end{multline}

\begin{theorem}[\cite{KW1}]\label{P-thm}
  Any partition $\tau$ may be transformed into a formal linear
  combination of planar partitions by repeated application of
  \ref{part-trans}, and the resulting linear combination does not
  depend on the choices made when applying \ref{part-trans}, so
  that we may write
  $$\tau \to \sum_{\text{\rm planar partitions $\sigma$}} \p_{\sigma,\tau} \sigma.$$
\end{theorem}

\subsection{Minimality and electrical equivalence}
\old{
A circular planar graph is said to be {\bf well-connected} if $\det L_A^B>0$
for all non-interlaced sets of nodes $A,B$. 
By Theorem \ref{generalminor}, this will be true if $\G$ supports, 
for each non-interlaced pair $A,B$, 
a grove which pairs the vertices of $A$ and $B$.
Any such grove will have as a subset a set of disjoint simple paths
pairing the vertices of $A$ and $B$, and conversely, any 
set of disjoint paths pairing $A$ and $B$ can be completed to such a grove
by adding edges. 
Thus the condition that $\G$ be well-connected can be translated
into a condition about pairing sets of vertices with disjoint simple paths. 

In section \ref{reconsection} we discuss the minimal sets of inequalities 
which imply well-connectedness.
We first discuss the medial graph and indicate how it plays a role in 
determining well-connectedness and equivalence of minimal networks.
}

\subsubsection{Medial graph}

If $\G$ is a circular planar network or network on a surface,
the {\bf medial graph} of $\G$ is the graph with a vertex $v(e)$ for every edge $e$
of $\G$, and an edge connecting $v(e)$ and $v(e')$ if $e$ and $e'$ share a 
vertex and are
adjacent edges (in cyclic order) out of this vertex. The medial graph is regular of 
degree $4$.  We usually cut the edges of the medial graph which
separate the nodes from $\infty$, so that the medial graph has two 
half-edges (called ``{\bf stubs}") adjacent to each node, see Figure \ref{medialgraph}.

\begin{figure}[htbp]
\center{\includegraphics[width=2.5in]{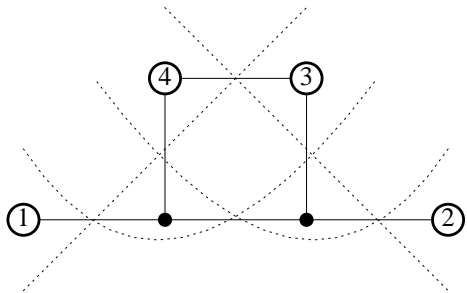}}
\caption{\label{medialgraph} The medial graph (dotted) for a 
$4$-node network (solid).}
\end{figure}
A {\bf strand} of the medial graph $M$ is a maximal path in the medial graph
which continues ``straight" at each vertex of $M$, that is, turns neither
left nor right. There is a strand starting at each stub and ending at another stub,
and possibly other strands which form closed loops.

If $\G$ is drawn in a disk, with the nodes on the bounding circle,
the strands of a medial graph are the boundaries of a cell decomposition
of the disk. The cells come in two types, corresponding to vertices and faces
of $\G$ (between two adjacent nodes there is a cell corresponding to an 
exterior face of $\G$). 

A nice property of medial strands is that, for minimal
graphs (see the next section), they are boundaries
of the support of equilibrium potentials. Let us explain.

\begin{lemma}\cite{CdV}
Suppose $\G$ is minimal.
Let $a,b$ be stubs connected by a medial strand. This medial strand
divides the disk into two parts $U_+,U_-$. 
Then there is an equilibrium potential $f$ (a function harmonic on $I$ and with Dirichlet boundary conditions on $\No$),
not identically zero, with harmonic conjugate $f^*$
with the property that both $f$ and $f^*$ are identically zero on (cells in) $U_-$.
\end{lemma}

For an example see Figure \ref{equilpot}.
Furthermore we have
\begin{lemma}\label{equilpotlemma}\cite{CdV}
If $G$ is minimal and $a,b$ are two stubs, let $C_+$ and $C_-$ be the two 
arcs of the boundary of the disk separated by $a$ and $b$.
Then $a$ and $b$ are connected by a medial strand
in $\G$ if and only if there are equilibrium potentials $f_+$
and $f_-$ and conjugates respectively 
$g_+,g_-$ such that $f_+,g_+$ are zero on nodes and
exterior faces of $C_+$,
and $f_-,g_-$ are zero on nodes and exterior faces of $C_-$. 
\end{lemma}

\begin{figure}[htbp]
\center{\includegraphics[width=2.5in]{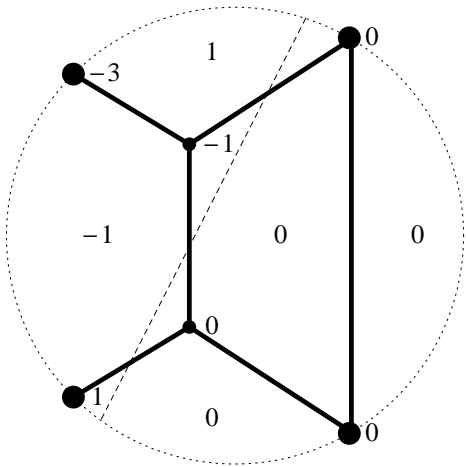}}
\caption{\label{equilpot}A network (solid) with a strand (dashed) and an
equilibrium potential and dual which are zero on one side. In this example
all conductances
are $1$.}
\end{figure}

\subsubsection{Minimality and electrical transformations}\label{YDeltamoves}

A circular planar network is said to be {\bf minimal} if the medial strands
have the properties 
\begin{enumerate}
\item There are no closed loops.
\item A strand does not intersect itself.
\item No two strands intersect more than once.
\end{enumerate}

The reason we define minimal graphs is that every circular planar network
is equivalent, in the sense of having the same response matrix, to 
a minimal circular planar network. Moreover minimal graphs are precisely
those in which the reconstruction has a unique solution \cite{CdV}.

In \cite{CGV} it was shown
that any circular planar network can be converted into a minimal 
circular planar network by a sequence of local transformations called
electrical transformations.
An {\bf electrical transformation} is a local rearrangement of the graph
and conductances
of one of the types shown in Figure \ref{electricalmoves}.
They consist of:
\begin{enumerate}
\item Removing a ``dead branch" (remove a non-node vertex of degree $1$ and the edge connecting it to the rest of the graph) or a self-loop.
\item Replacing two conductors $c_1,c_2$ in series with a single conductor
of conductance $\frac{c_1c_2}{c_1+c_2}$
(as long as the common vertex is not a node).
\item Replacing parallel edges of conductances $c_1,c_2$ with a single edge
of conductance $c_1+c_2$.
\item the ``Y-Delta" move, also called star-triangle move: an internal vertex of degree three with edges
of conductance $c_1,c_2,c_3$ is replaced by a triangle with edges of
conductance $\frac{c_1c_2}{c_1+c_2+c_3},\frac{c_1c_3}{c_1+c_2+c_3},
\frac{c_2c_3}{c_1+c_2+c_3}$ as in the diagram.
\end{enumerate}
The reverse of a Y-Delta move is a Delta-Y move and is also allowed.
The inverse transformation takes conductances $c_1,c_2,c_3$ on the ``Delta"
to conductances
$$\frac{c_1c_2+c_1c_3+c_2c_3}{c_1},\frac{c_1c_2+c_1c_3+c_2c_3}{c_2},\frac{c_1c_2+c_1c_3+c_2c_3}{c_3}$$
on the Y. For simplicity when we say ``Y-Delta move" we mean either
a Y-Delta or its reverse.
\begin{figure}
\center{\includegraphics[width=6in]{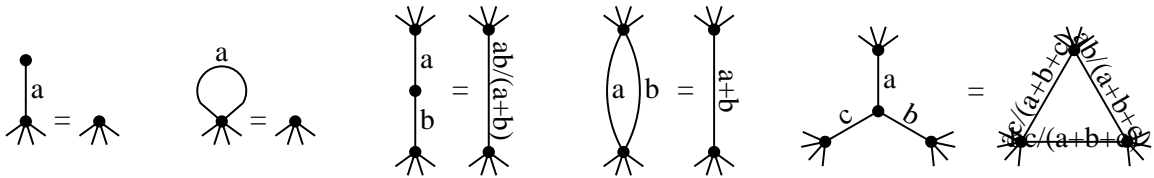}}
\caption{\label{electricalmoves}The electrical transformations.}
\end{figure}

We say that circular planar networks $\G,\G'$ on the same number of nodes
are {\bf topologically equivalent} if one can be converted to the other as graphs
(that is, ignoring conductances), using electrical transformations. 
We define circular planar networks
$\G,\G'$ to be {\bf electrically equivalent} if they have the same
response matrix. \label{elecequiv}

\subsubsection{Topological equivalence}
It is not hard to show that any circular planar network is topologically equivalent to a minimal
circular planar network, that is, can be converted to a minimal
network by electrical transformations. The idea is to simply chart how the moves act on medial strands:
Each of the first three types of moves decreases the number of crossings
of the medial strands (that is, the number of edges of the network). 
We isotope the strands, leaving their endpoints fixed,
until each strand
has no self-intersections and crosses each other strand at most once.
During this isotopy, each time a strand crosses an intersection of other
strands we perform a Y-Delta move on the graph; 
other singularities encountered will decrease the number of
crossings and so are finite in number.

This argument shows that for the purposes of determining
topological equivalence one can assume $\G$ and $\G'$ 
are minimal. 
What is a bit harder to show is that the ``minimization" process leads to a
unique minimal graph, up to Y-Delta moves. That is, if $\G'$ and $\G''$
are minimal graphs obtained from the same graph $\G$ then
$\G'$ can be converted to $\G''$ using only
Y-Delta moves (and Delta-Y moves). See Theorem \ref{topequiv} below.

For a given minimal circular planar graph $\G$ with $n$ nodes, 
the medial graph has $2n$ stubs, one starting to the left and right of each
node. The medial paths connect these stubs in pairs.
Let $\pi$ denote the pairing
of the stubs by the medial paths; if we label the stubs from $1$ to $2n$ in cclw
order then we can think of $\pi$ as a fixed-point free involution
of $\{1,2,\dots,2n\}$.

\begin{thm}[\cite{CdV}]\label{topequiv}
Two minimal graphs $\G,\G'$ are topologically equivalent
if and only if they have the same medial strand stub involution.
\end{thm}

\begin{proof}
If two minimal graphs have the same medial
strand involution then one can isotope the strands without forming any
new crossings so the two graphs are isomorphic: for example
put the nodes and stubs in general position on a circle respecting their circular order; now isotope the strands to straight chords across the disk. 

Conversely, by Lemma \ref{equilpotlemma}
the response matrix of a minimal graph $\G$ determines its stub involution, 
since the supports
of equilibrium potentials and their harmonic conjugates are determined by $L$.
\end{proof}

\subsubsection{Well-connected and non-well-connected networks}

Recall that a network is well-connected if all non-interlaced minors of $L$ are strictly positive.
Equivalently, for any non-interlaced subsets of nodes $A,B$, there is a grove
connecting points of $A$ to those of $B$ in pairs. (If $|A|=|B|=k$ one can just look for a set of $k$
pairwise disjoint paths connecting points of $A$ to those of $B$; such a set is easily completed
to a grove by adding extra edges.)

What stub involution corresponds to a well-connected network?
It is not hard to construct well-connected graphs.
One nice family is illustrated in Figure \ref{stdgraphs}.
\begin{figure}
\center{\includegraphics[width=5in]{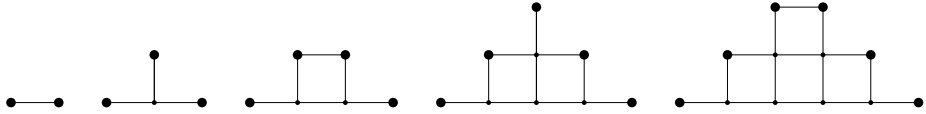}}
\caption{\label{stdgraphs}The standard networks $\Gamma_k$ for $k=2,\dots,6$.
These are well-connected.}
\end{figure}
We leave it to the reader to show that 
these graphs are well-connected, and have stub involution
$\pi_{\text{well}}$ which pairs $i$ with $n+i$ for each $i$.
That is, each stub is paired with the diametrically opposite stub. 

\begin{cor} For a minimal graph $\G$ on $n$ nodes, the following are equivalent:
\begin{enumerate}\item $\G$
is well-connected.
\item Its stub involution
is $\pi_{\text{well}}$. 
\item $\G$ is topologically equivalent to $\Gamma_n$.
\end{enumerate}
\end{cor}

Another nice family of well-connected graphs which are in addition
circularly symmetric (for $n$ odd) and nearly circularly symmetric 
(for $n$ even) is given in Figure \ref{circwellconn}. See Figure \ref{rectwellconn} for a third family.
\begin{figure}
\center{\includegraphics[width=6in]{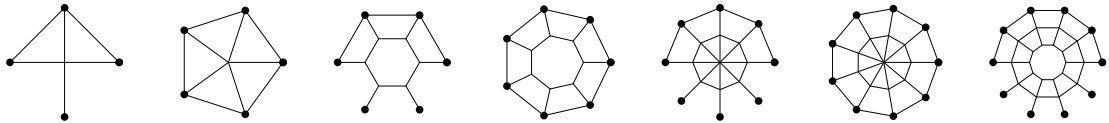}}
\caption{\label{circwellconn}Some (almost) circularly symmetric well-connected graphs.}
\end{figure}

\begin{figure}
\center{\includegraphics[width=4in]{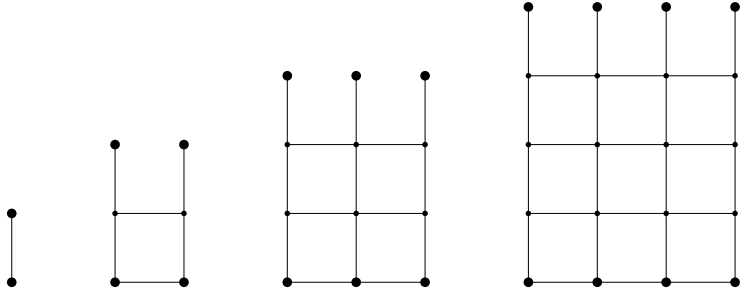}}
\caption{\label{rectwellconn}A different family of well-connected graphs for even $n$.}
\end{figure}

What about non-well connected graphs?
Is there any structure to the set of topological equivalence classes of 
$n$-node networks? See below.

\subsubsection{Electrical equivalence}
We discussed topological equivalence above.
Regarding electrical equivalence, 
we have
\begin{thm} Electrical transformations do not change the response matrix.
\end{thm}

\begin{proof}
It suffices to show that electrical transformations preserve current flow.
More precisely, the current flow on the ``before" network will be
equal to the current flow
on the ``after" network if defined appropriately on any added edges.

This is clear for transformations of type $1$ (removing a dead branch or self-loop)
since these edges have no current flow.
For a type $3$ move (parallel reduction) the current flow on $c_1$ is 
$c_1 df(e)$ where $f$ is the corresponding potential on the vertices.
Similarly the current flow across $c_2$ is $c_2df(e)$. The current flow
on the combined edge is $(c_1+c_2)df(e)$ which is the sum of the current
flows on the individual edges. 

For a type $2$ (series) transformation, the relevant equation is
$c_1df(e_1)=c_2df(e_2)$ since no current is lost at the center vertex.
This implies that 
$$c_1df(e_1)=\frac{c_1c_2}{c_1+c_2}(df(e_1)+df(e_2)),$$
and $df(e_1)+df(e_2)$ is the potential drop across the combined edges.

A similar computation works for the Y-Delta move. When converting from a Y
to a Delta, the current along the edges of the Y directed towards
the central vertex is $c_1df_1,c_2df_2,c_3df_3$ (where the $df_i$
are the potential drops) which must add to zero since no
current is lost at the central vertex:
$$c_1df_1+c_2df_2+c_3df_3=0.$$
A short computation then shows that
\begin{eqnarray*}c_1df_1 &=& C_{12}(df_1-df_2)+C_{13}(df_1-df_3),\\
c_2df_2 &=& C_{12}(df_1-df_2)+C_{13}(df_1-df_3),\\
c_3df_3 &=& C_{12}(df_1-df_2)+C_{13}(df_1-df_3)
\end{eqnarray*}
for arbitrary $df_1,df_2,df_2$ when and only when $C_{ij}=\frac{c_ic_j}{c_1+c_2+c_3}$.
\end{proof}

As a consequence of this theorem,
for determining the space of response matrices of planar
networks we need only consider minimal networks.

\subsection{Reconstruction and the space of networks}\label{reconsection}

Given a matrix which satisfies the conditions and inequalities
of Theorem (\ref{CdVthm}),
is it the response matrix of a circular planar network? How can one determine
the conductances from $L$?
This is the {\bf Reconstruction Problem} or {\bf EIT problem}. 
It is closely related
to the {\bf Electrical Equivalence Problem}: what sets of minimal networks
are electrically equivalent?

\subsubsection{Algorithm}

Both \cite{CIM} and \cite{CdV} gave an iterative algorithm 
for reconstruction on any minimal graph.
The idea is as follows. 
Any network is equivalent to one which has either an edge connecting
adjacent nodes or a node of degree $1$ (take a strand $ab$
with the property
that there are no strands completely contained on one side $U_+$ of it;
using Y-Delta moves, 
move all crossings in $U_+$ of other strands to the other side $U_-$. The node 
adjacent to $a$ and inside $U_+$ will either have degree $1$ or be connected
to its adjacent node on the other side of $a$). 

One can now compute the conductance on this edge $e$ from $L$:
find the potential with values
zero on side $U_+$ of strand $ab$. Then the conductance $c_e$ is easily obtained
from the current.

Removing the edge $e$ (if it connected nodes) or contracting the edge $e$
(if it is adjacent to a degree-$1$ node)
results in a new network which is still minimal, and the new response matrix
is a simple rational function of the old response matrix. 

This algorithm shows that the conductances are \emph{rational functions of the $L_{ij}$}. 
However it is not so easy to 
use this to give an explicit formula for the conductances
as a function of the response matrix.
In \cite{KW2} an explicit formula is given 
for the reconstruction problem for the standard graphs $\Gamma_n$
of Figure \ref{stdgraphs},
which expresses $c_e$ as a ratio of polynomials in the $L_{ij}$
which are Pfaffians of skew-symmetric matrices formed from
$L$. This can be considered in some sense the best possible result for
$\Gamma_n$ since the polynomials of the $L_{ij}$ 
are shown there to be irreducible in general.
However this does not preclude the possibility that for other
well-connected networks there is a simpler formula: see \cite{KW4}.

\old{
By allowing the conductances on the edges of $\Gamma_n$ to vary in
$(0,\infty)$ we get all response matrices of well-connected networks,
and the map from conductances to matrices is injective.
Thus the space of well-connected response matrices is 
homeomorphic to an open ball of dimension $\begin{pmatrix}n\\2\end{pmatrix}$,
in fact birational with $\R_+^{n(n-1)/2}$.
}

\subsubsection{Partial order on networks}

By allowing some conductances to go to zero and others to tend to $\infty$
a network degenerates (topologically, we remove edges
of conductance zero and contract edges of conductance $\infty$).

There is a partial order on networks under this degeneration.
We define $\G\le \G'$ if a network equivalent to $\G$ can be obtained from a network equivalent to $\G'$ by contraction and deletion of edges.

One can show that
every minimal network is a minor (in the sense of contraction and/or 
deletion of edges)
of a well-connected network. So the well-connected network is the unique 
maximal element in this partial order. If we allow disconnected networks,
the networks which are minimal are those whose stub involution is a 
non-crossing matching; see Figure \ref{partialorder3}.

\begin{lemma}\label{sequence} If $\G,\G'$ are minimal circular planar networks
and $\G\le \G'$ then there is a sequence
$\G_1\le \G_2\le\dots\le\G_k=\G'$ of minimal networks
where $\G_1$ is equivalent to $\G$ and
$\G_i$ is obtained from a network equivalent to $\G_{i+1}$ by deletion
or contraction of a single edge.
\end{lemma}
\begin{proof} First suppose that $\G$ is obtained from $\G'$ by deleting
a single edge $e$ (the case of contraction is equivalent by duality).
Here is how one minimizes the resulting network:
let $a,b,c,d$ in cclw order 
be the stubs of the strands $ac$ and $bd$ crossing at edge $e$.
Let $U_1,U_2,U_3,U_4$ be the four arcs of the circle separated by $a,b,c,d$,
so that $U_1$ is the arc from $a$ to $b$, etc. 
Using Y-Delta moves we can first comb strands so that
strands connecting stubs from $U_1$ to $U_3$ cross ``left" of $e$,
that is, do not enter the region delimited by the path $bec$, 
and strands connecting stubs from $U_2$ and $U_4$ cross ``below" $e$,
that is, do not enter the region delimited by the path $ced$.
See the Figure \ref{abcd}.
Now we resolve the crossing at $e$, creating a strand connecting $a$ and $b$
and another connecting $c$ and $d$. The strand $ab$ forms a bigon with any
strand from $U_2$ to $U_4$. To minimize
the resulting network we resolve all crossings of these strands with the 
semi-strand from $b$ to $e$ as indicated in Figure \ref{abcd}. The resulting network is now 
minimized (and any minimization results in a network equivalent to this one).
\begin{figure}[htbp]
\center{\includegraphics[width=2in]{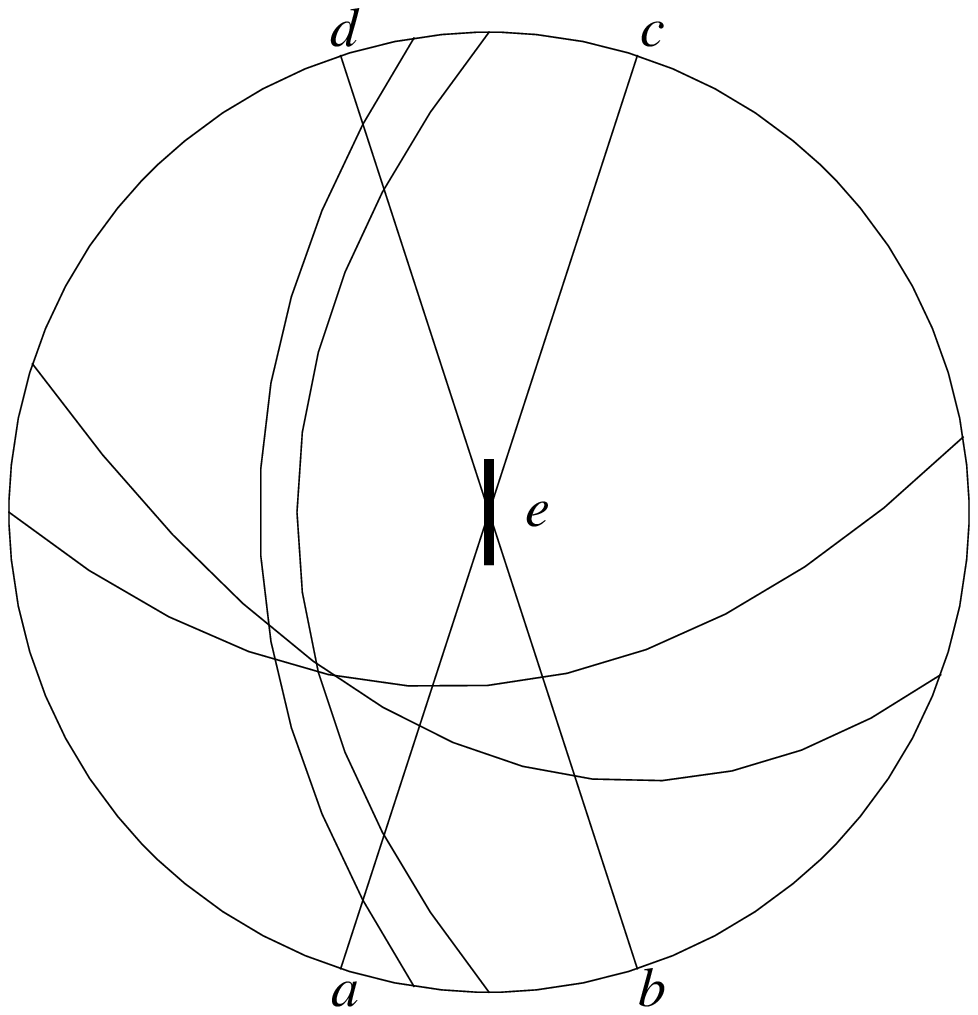}\hskip1in\includegraphics[width=2in]{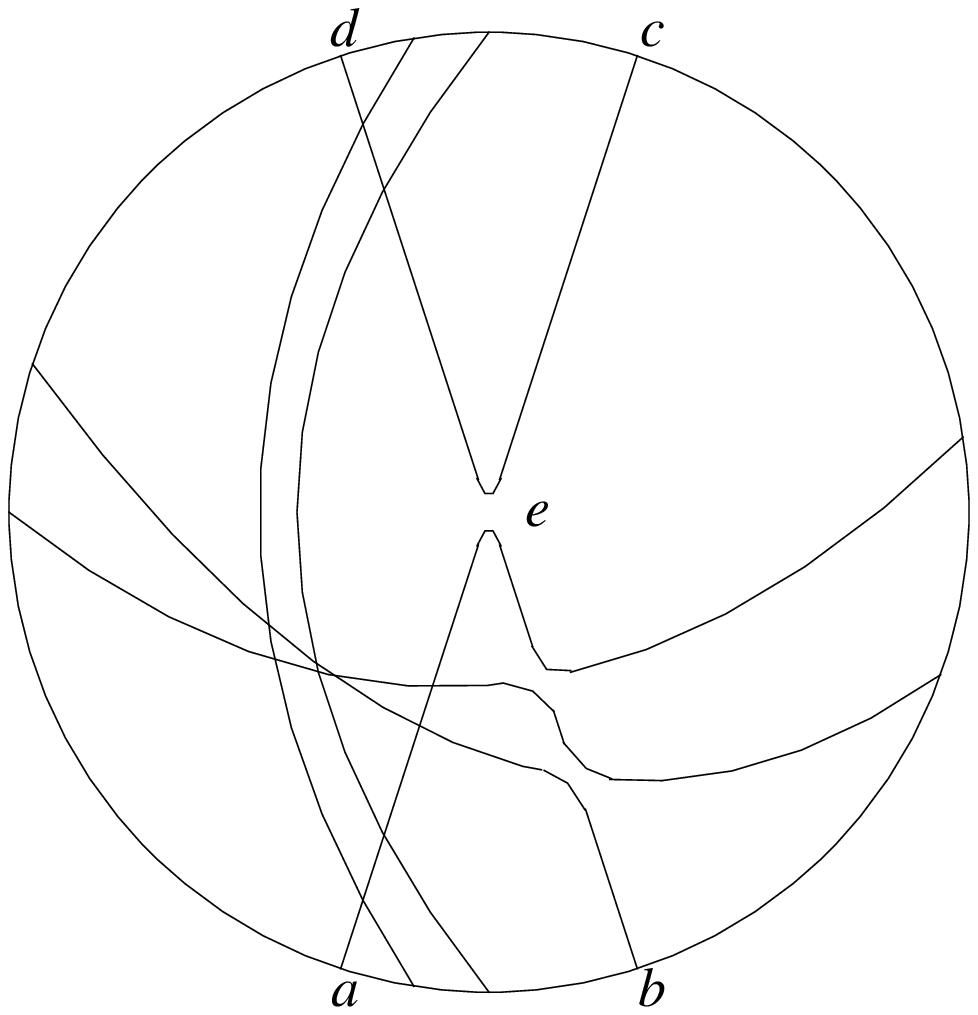}}
\caption{\label{abcd}When a conductance $e$ goes to zero,
comb strands left and down (left figure); then resolve strand crossings
between $b$ and $e$ (right figure).
}
\end{figure}
If we perform this sequence of resolutions in order
of the crossings along the semistrand $be$ from the boundary towards the center, 
(resolving the crossing at $e$ last) 
the intermediate networks are all minimal.
This completes the proof in the case when $\G$ is obtained from $\G'$
by resolution of a single crossing.

In the general case, 
we contract and delete multiple edges of $\G'$, then minimize
to get $\G$. Take a path in the space of conductances which sends the 
desired conductances to zero or infinity, one after another, in any order.
Let $L(t)$ be the corresponding path in the space of response matrices.
When we minimize the resolution of a single crossing as above,
the resulting $L$ matrix does not change (in the sense that no additional minors
become zero), since minimization leaves $L$ invariant.
Thus the sequence of resolutions is replaced with  
a sequence of ``minimal'' resolutions, that is, resolutions which preserve minimality. 
\end{proof}

This partial order on networks
induces a partial order on fixed-point free involutions of 
$\{1,\dots,2n\}$.  This partial order is graded by the number of crossings
in the involution, that is, the number of pairs $i<j<k<\ell$ with
$i$ paired with $k$ and $j$ paired with $\ell$. Moving one level down in the partial order
corresponds to resolving exactly one crossing which is adjacent to the boundary
(in the sense that there is a path in the disk from the crossing to the boundary circle which
does not meet any strand). Algebraically, moving up one level corresponds
to the operation of conjugating the involution by a transposition in such a way as to 
increase the number of crossings by $1$.

See Figure \ref{partialorder3} for the Hasse diagram of the partial order in the 
case $n=3$. 
\begin{figure}[htbp]
\center{\includegraphics[width=4.5in]{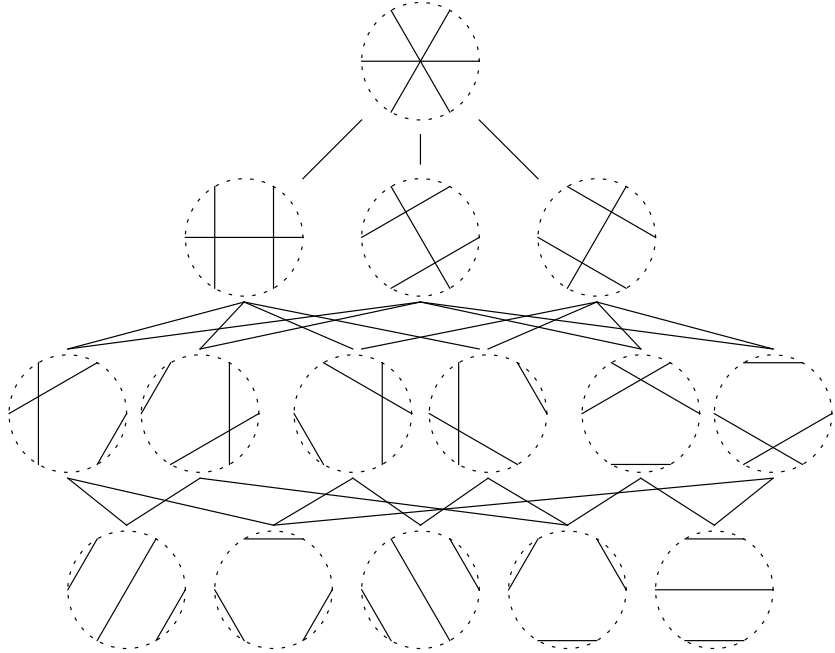}}
\caption{\label{partialorder3}The partial order on circular planar networks
with three nodes. Only the medial graphs are shown.}
\end{figure}

This partial order describes the cell structure of the space of networks.

\subsubsection{Minimal sets of inequalities}

What is the structure of the set of response matrices of all circular 
planar networks?
The set $\Omega$ of response matrices of circular planar networks on $n$ nodes
is a closed subset of $\R^{n(n-1)/2}$ if we
use the coordinates $\{L_{ij}\}_{n\ge i>j\ge 1}.$ It is a semi-algebraic set,
that is, defined by a set of algebraic inequalities $\det L_A^B\ge 0$
for the non-interlaced subsets $A,B$ of nodes. Its interior $\Omega^+$ is defined
by using strict inequalities, and corresponds to the well-connected networks.

The number of inequalities defining $\Omega$ is exponential in $n$. 
However there is a smaller set of $n(n-1)/2$ 
inequalities which defines $\Omega^+$. See below.

If we consider a topological equivalence class 
(defined by a stub involution $\pi$) of graphs,
the response matrices of networks supported on these graphs form a subset
$\Omega_\pi\subset\Omega$, which is a subset of the boundary of $\Omega$
if $\pi\ne\pi_{\text{well}}$. 
$\Omega$ has the structure of a cell-complex, in which 
the $\Omega_\pi$ are the cells, see Lam and Pylyavskyy \cite{LP1}. Is each cell of dimension $k$ defined by
$k$ minor inequalities (and $n(n-1)/2-k$ equalities)?
 The answer is probably yes, but at the moment our understanding
is limited. A similar situation where the cell structure is explicitly worked out is
Postnikov \cite{Postnikov} who deals with totally
positive/totally nonnegative matrices and the totally nonnegative Grassmannian.

Let us discuss here only the minimal sets of inequalities defining
the set $\Omega^+$. 
Somewhat remarkably, there are many different sets $\begin{pmatrix}n\\2\end{pmatrix}$
minor determinants of $L$ whose positivity implies the positivity
of all non-interlaced minors. The situation resembles that of a cluster algebra
but so far we have been unable to find the relevant cluster structure.

One easy-to-remember set of $\begin{pmatrix}n\\2\end{pmatrix}$ minors 
are the {\bf central minors}, defined as follows. Suppose first $n$ is odd. 
Define $M_{i,1}$ for $1\le i\le (n-1)/2$ to be the minor
$L_A^B$ with $A=\{1,2,\dots,i\}$ and $B=\{(n-1)/2+1,(n-1)/2+2,\dots,(n-1)/2+i\}$.
Let $M_{i,j}$ be the minor obtained from $M_{i,1}$ by rotating the indices
cyclically by $j-1$, so that $A=\{j,j+1,\dots,j+i\}$ and likewise for $B$.

These minors are called central minors since, if we arrange the $n$ points 
evenly spaced on a circle, and connect each element of $A$ to the corresponding
point of $B$ opposite it, we get a set of $i$ parallel chords which are
``central" in the sense that they
are as close to a fixed diagonal of the circle as possible while remaining disjoint
from each other. See Figure \ref{centralminors} for an example.
\begin{figure}
\center{\includegraphics[width=5in]{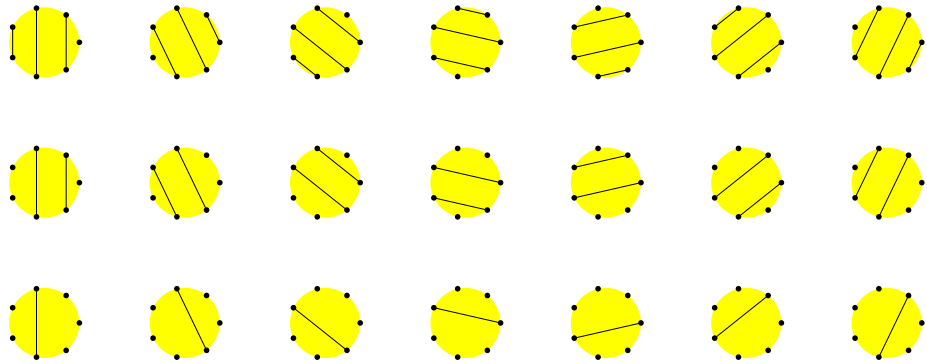}}
\caption{\label{centralminors}The $21$ central minors for a well-connected network on $7$ nodes.}
\end{figure}
In the case $n$ is even we modify the above; define $M_{i,j}$ for $i$ even as before, but for $i$ odd we have, for each diagonal, two choices of positioning
of the chords to make them ``most central"; choose one of these arbitrarily.
These ``off-center" minors are part of our set of central minors.
See Figure \ref{centralminorsix} for a natural choice.

\begin{figure}
\center{\includegraphics[width=5in]{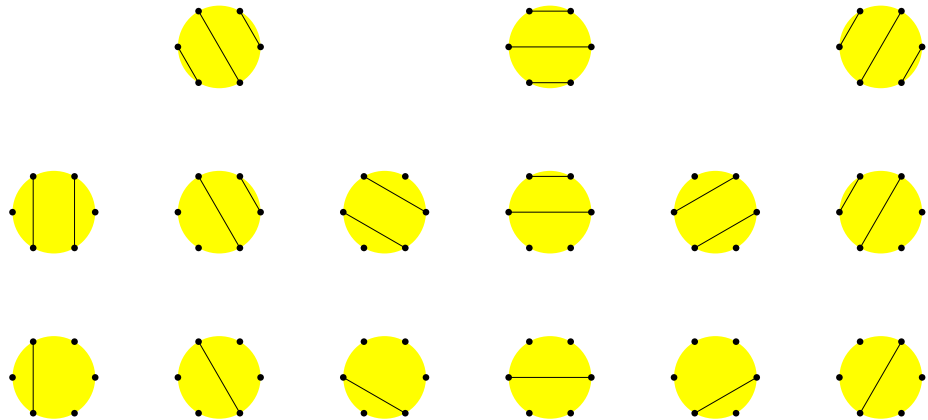}}
\caption{\label{centralminorsix}The $15$ central minors for a well-connected network on $6$ nodes.}
\end{figure}
\begin{thm}[\cite{KW4}]\label{centralminorthm} If the central minors $\{M_{i,j}\}$ are positive then all
noninterlaced minors are positive.
\end{thm}

The following proof is lifted directly from \cite{KW4}.
\begin{proof} This follows using two identities for minors.

Let $U$ be a matrix, and let $a,b,c$ index some of its columns, and $z,d$ index some of its rows.  Then it is elementary that
$$
0=
\begin{vmatrix}
U_{z,a}&U_{z,b}&U_{z,c}\\
U_{z,a}&U_{z,b}&U_{z,c}\\
U_{d,a}&U_{d,b}&U_{d,c}\\
\end{vmatrix}
=
U_{z,a}
\begin{vmatrix}
U_{z,b}&U_{z,c}\\
U_{d,b}&U_{d,c}\\
\end{vmatrix}
-U_{z,b}
\begin{vmatrix}
U_{z,a}&U_{z,c}\\
U_{d,a}&U_{d,c}\\
\end{vmatrix}
+U_{z,c}
\begin{vmatrix}
U_{z,a}&U_{z,b}\\
U_{d,a}&U_{d,b}\\
\end{vmatrix}.
$$
Suppose $U$ is invertible, and let $M$ denote its inverse.
Let $r_1,\dots,r_k$ denote the row indices of $M$ (and column indices of $U$),
and $c_1,\dots,c_k$ denote the column indices of $M$ (and row indices of $U$).
Dividing through by $(\det U)^2$ and using Jacobi's identity relating
minors of a matrix to minors of its inverse, we obtain
$$ 0 = \det(M_{i,j})^{i\neq a}_{j\neq z} \det(M_{i,j})^{i\neq b,c}_{j\neq z,d} - \det(M_{i,j})^{i\neq b}_{j\neq z} \det(M_{i,j})^{i\neq a,c}_{j\neq z,d} + \det(M_{i,j})^{i\neq c}_{j\neq z} \det(M_{i,j})^{i\neq a,b}_{j\neq z,d}. $$

Since the column $z$ is always excluded, we may as well suppose that $M$ is a $k\times (k-1)$ matrix.  Then
$$ \det(M_{i,j})^{i\neq b}_{j} \det(M_{i,j})^{i\neq a,c}_{j\neq d} = \det(M_{i,j})^{i\neq a}_{j} \det(M_{i,j})^{i\neq b,c}_{j\neq d}  + \det(M_{i,j})^{i\neq c}_{j} \det(M_{i,j})^{i\neq a,b}_{j\neq d}. $$

For example, taking $M=L^{9,8,7,6,5}_{1,2,3,4}$ and $a,b,c,d=9,8,5,4$, we have
$$
\Big|L^{9, 7, 6, 5}_{1, 2, 3, 4}\Big|  =
\frac{\Big|L^{8, 7, 6, 5}_{1, 2, 3, 4}\Big| \, \Big|L^{9, 7, 6}_{1, 2, 3}\Big|
  +
\Big|L^{9, 8, 7, 6}_{1, 2, 3, 4}\Big| \, \Big|L^{7, 6, 5}_{1, 2, 3}\Big|}{\Big|L^{8, 7, 6}_{1, 2, 3}\Big|}.
$$
We can denote this pictorially as
$$
\disk{9 7 6 5}{1 2 3 4}{9}  =
\frac{\disk{8 7 6 5}{1 2 3 4}{9} \disk{9 7 6}{1 2 3}{9}
  +
\disk{9 8 7 6}{1 2 3 4}{9}  \disk{7 6 5}{1 2 3}{9}}{\disk{8 7 6}{1 2 3}{9}}.
$$
We call transformations of this type the ``jaw move''.

For a given non-interleaved determinant interspersed with at least one isolated node, we can take $b$ to be one of the interspersed isolated nodes, and $a$ and $c$ to be the first and last of the nodes on the same side as $b$, and $d$ to be either first or last node on the other side as $b$.  With this choice of $a,b,c,d$, the jaw move expresses the original determinant as a positive rational function of ``simpler'' non-interleaved determinants, where a determinant is simpler if it has fewer strands, or else the same number of strands but fewer interspersed isolated nodes.

By repeated application of the jaw rule, any non-interleaved determinant can be expressed as a positive rational function of non-interleaved contiguous determinants.

The other identity that we need is based on Dodgson condensation (and is
also called the Desnanot-Jacobi identity).  It can be derived in a similar way as the jaw move, with the elementary starting identity
$$
U_{c,a} U_{d,b}
=
\begin{vmatrix}
U_{c,a}&U_{c,b}\\
U_{d,a}&U_{d,b}\\
\end{vmatrix}
 + U_{c,b} U_{d,a}.
$$
Dividing through by $(\det U)^2$ and using Jacobi's identity, we obtain
$$ \det(M_{i,j})^{i\neq a}_{j\neq c} \det(M_{i,j})^{i\neq b}_{j\neq d} =
 \det(M_{i,j})^{i}_{j} \det(M_{i,j})^{i\neq a,b}_{j\neq c,d} +
 \det(M_{i,j})^{i\neq b}_{j\neq c} \det(M_{i,j})^{i\neq a}_{j\neq d}. $$
For example, taking $M=L^{8,7,6}_{1,2,3}$ and $a,b,c,d=8,6,3,1$, we obtain
$$
\Big|L^{8,7}_{1, 2}\Big|  =
\frac{\Big|L^{8, 7, 6}_{1, 2, 3}\Big| \, \Big|L^{7}_{2}\Big|
  +
\Big|L^{7,6}_{1, 2}\Big| \, \Big|L^{8,7}_{2, 3}\Big|}{\Big|L^{7,6}_{2, 3}\Big|}.
$$
which we can denote pictorially as
$$
\disk{8 7}{1 2}{9}  =
\frac{\disk{8 7 6}{1 2 3}{9} \disk{7}{2}{9}
  +
\disk{7 6}{1 2}{9} \disk{8 7}{2 3}{9}}{\disk{7 6}{2 3}{9}}.
$$
We call transformations of this type the ``condensation move''.

If both sides of the determinant are contiguous, but the crossings are off-center, then we can pick $a$ and $d$ to be the paired nodes that are most off-center, and $b$ and $c$ to be the nodes next to the crossing but not in the crossing (and closer to the center).  Then after a condensation move, the original determinant is expressed as a positive rational function of contiguous crossings that are strictly more central.  By repeated application of such condensation moves, any contiguous determination may be expressed as a positive rational function of central contiguous crossings.

The remaining case to check is for $n$ even, with the (minimally) off-center contiguous crossings come in pairs, of which only one is included in the base cluster.  Using a condensation move we can express a minimally off-center contiguous crossing in terms of its opposite minimally off-center contiguous crossing and central contiguous crossings as shown below
$$
\disk{1 2}{7 6}{8} = \frac{\disk{1 2}{6 5}{8} \ \disk{2 3}{7 6}{8} + \disk{2}{6}{8} \ \disk{1 2 3}{7 6 5}{8}}{\disk{2 3}{6 5}{8}}
$$

\end{proof}

\subsection{The Jacobian} 
A remarkable property of the 
birational map from conductances to the (appropriate set of)
$L$ matrix minors for minimal networks is that
the Jacobian is $\pm1$:
\begin{thm}[\cite{KW4}]
$$\det\left(\frac{\partial\log M_i}{\partial \log c_j} \right)= \pm 1.$$
\end{thm}

In other words the volume form $\prod_{e} \frac{dc_e}{c_e}$
on the space of conductances 
is mapped to $\pm1$ times the volume form
$\prod_i\frac{dM_i}{M_i}$ on the space of $L$-matrix minors.
Note that the matrix entries have a probabilistic interpretation:
each $M_i$ can be written $M_i=\Pr(\pi)/\Pr(\text{unc})$ for some noninterlaced pairing
$\pi$. Thus 
$$\frac{\partial\log M_i}{\partial \log c_j}=\frac{c_j}{Z(\pi)}\frac{\partial Z(\pi)}{\partial c_j}-\frac{c_j}{\Zunc}\frac{\partial \Zunc}{\partial c_j}$$ is the difference in the probability that
edge $e_j$ is in a random $\pi$-pairing, minus the probability that $e_j$
is in a random uncrossing.

The proof of this theorem relies of techniques developed in \cite{Postnikov} and \cite{Talaska}
and is too long to give here. Let us just give an example, for a
well-connected network on $3$ nodes, the Y of Figure \ref{Ynetwork}. 
The central minors in terms of the conductances are
$L_{ij}=\frac{c_ic_j}{c_1+c_2+c_3}.$
The above Jacobian matrix (with columns in the order $L_{12},L_{13},L_{23}$) is 
$$\begin{pmatrix}1-\frac{1}{c_1+c_2+c_3}&1-\frac{1}{c_1+c_2+c_3}&-\frac{1}{c_1+c_2+c_3}\\
1-\frac{1}{c_1+c_2+c_3}&-\frac{1}{c_1+c_2+c_3}&1-\frac{1}{c_1+c_2+c_3}\\
-\frac{1}{c_1+c_2+c_3}&1-\frac{1}{c_1+c_2+c_3}&1-\frac{1}{c_1+c_2+c_3}\\
\end{pmatrix}$$
whose determinant is $1$.

\section{Networks on surfaces with nontrivial topology}

Some of the material in the preceding section on circular planar networks
has been extended to the case of annular networks and more generally networks
on surfaces. However the theory is not as complete at present as in the
case of circular planar graphs. Moreover beyond the annulus and torus the 
theory is for the most part nonexistent. 

We can nonetheless ask the same questions as arose in the planar case:
\begin{enumerate}
\item What is the natural notion of response matrix?
\item For which networks can we reconstruct
the conductances from the response matrix?
\item What is the structure of the space of response matrices?
\item What natural combinatorial objects are relevant?
\end{enumerate}

Our answers, briefly, are as follows. We elaborate on these in the following sections.
\begin{enumerate}
\item We use (the Schur reduction of) the bundle Laplacian for a flat
$\C^*$-connection or a flat $\SL$-connection on the network.
\item We can conjecturally reconstruct conductances for minimal networks
on the annulus. For other surfaces reconstruction is not in general possible,
unless we are given more data: we (conjecturally)
need to provide ``spectral" data as well.
\item We don't know what the structure of the space of response matrices is,
even for the annulus. However we have a description in the case when there
are no nodes, for the annulus and torus.
\item The natural combinatorial objects are ``cycle-rooted spanning forests": forests
in which each component has a unique (topologically nontrivial) cycle.
\end{enumerate}

\subsection{Vector bundle Laplacian}

In the case of a network on a surface it is natural
to consider not the standard Laplacian but a more general Laplacian operator, the vector bundle Laplacian,
which depends on a connection on the bundle (see definitions below).
This is because the standard Laplacian does not contain very much
information about the topology of the surface on which the network is 
embedded...in particular the reconstruction problem is not solvable in general
(as an example consider the network on an annulus with two vertices, both nodes, and two edges joining them
with different homotopy classes). 
From the bundle Laplacian we construct a response matrix which contains more
information than the one constructed from the standard Laplacian,
and conjecturally allows reconstruction for minimal networks.

Let us first define vector bundles and connections on a network.

\subsubsection{Vector bundles and connections}
Given a fixed vector space $W$, 
a {\bf $W$-bundle}, or simply a {\bf vector bundle} on a network $\G$ is the choice of a vector space $W_v$ isomorphic to $W$ for
every vertex $v$ of $\G$. A vector bundle can be identified with the vector space 
$W_{\G}:=\oplus_{v}W_v \cong W^{|\G|}$, called the {\bf total space} of the bundle.
A {\bf section} of a vector bundle is an element of $W_{\G}$.

A {\bf connection} $\Phi$ on a $W$-bundle
is the choice for each edge $e=vv'$ of $\G$ of an isomorphism $\phi_{vv'}$
between the corresponding vector spaces $\phi_{vv'}:W_v\to W_{v'}$,
with the property that $\phi_{vv'}=\phi_{v'v}^{-1}$. This isomorphism
is called the {\bf parallel transport} of vectors in $W_v$ to vectors in $W_{v'}$. Two connections $\Phi,\Phi'$ are said to be
{\bf gauge equivalent} if there is for each vertex an isomorphism
$\psi_v:W_v\to W_v$ such that the diagram
$$\begin{CD}W_v @>\phi_{vv'}>> W_{v'}\\
@VV\psi_vV @VV\psi_{v'}V\\
W_v@>\phi'_{vv'}>> W_{v'}
\end{CD}
$$
commutes. In other words $\Phi'$ is just a base change of $\Phi$.
Note that the connection has nothing to do with the conductances.

Given an oriented  cycle $\gamma$ in $\G$ starting at $v$, the {\bf monodromy} of the connection around $\gamma$ is the 
element of $\text{End}(W_v)$ which is the product of the parallel transports 
around $\gamma$.
Monodromies starting at different vertices on $\gamma$ are conjugate,
as are monodromies of gauge-equivalent connections.

A {\bf line bundle} is a $W$-bundle where $W\cong\C$,  
the $1$-dimensional complex vector space. In this case given a connection
if we choose a basis for each $\C$
then the parallel transport along an edge is just multiplication by an element of $\C^*=\C\setminus\{0\}$.
The monodromy of a cycle is in $\C^*$ and does not depend
on the starting vertex of the cycle (or gauge).

\subsubsection{Flat connections}

A connection on a network embedded on a surface $\Sigma$ is said to 
be a {\bf flat connection} if the monodromy around faces of $\G$ is trivial.
This implies that the monodromy is trivial around loops on $\G$ which are 
null-homotopic as loops on $\Sigma$.

It is not hard to see that flat connections, modulo gauge equivalence,
are in bijection with homomorphisms of $\pi_1(\Sigma)$ into
the group of automorphisms of $W$. 

\subsubsection{The Laplacian}

The Laplacian $\Delta$ on a $W$-bundle with connection
$\Phi$ is the linear operator 
$\Delta: W_{\G}\to W_{\G}$ defined by
$$\Delta f(v)=\sum_{v'\sim v}c_{vv'}(f(v)-\phi_{v'v}f(v'))$$
where the sum is over neighbors $v'$ of $v$. 

Note that if the vector bundle is {\bf trivial}, in the sense that $\phi_{vv'}$ is the identity
for all edges, this is our the notion of graph Laplacian from (\ref{Lapdef})
(or more
precisely, the direct sum of $\dim W$ copies of the Laplacian).

As in the case of the standard network Laplacian we often 
think of $\Delta$ as a matrix $\Delta=(\Delta_{vv'})_{v,v'\in V}$ whose entries
are $-c_{vv'}$ times the parallel transport from $v'$ to $v$. In particular the conductance is acting as a scalar multiplication in $W$. 

Here is an example.
Let $\G=K_3$ the triangle with vertices $\{v_1,v_2,v_3\}$. 
Let $\Phi$ be the line bundle connection
with $\phi_{v_iv_j}=z_{ij}\in\C^*.$
Then in the natural basis, $\Delta$ has matrix
\begin{equation}\label{3vexample}
\Delta=\left(\begin{matrix}c_{12}+c_{13}&-c_{12}z_{12}&-c_{13}z_{13}\\
-c_{12}z_{12}^{-1}&c_{12}+c_{23}&
-c_{23}z_{23}\\-c_{13}z_{13}^{-1}&-c_{23}z_{23}^{-1}&c_{13}+c_{23}\end{matrix}\right).
\end{equation}

\subsubsection{Edge bundle}

One can extend the definition of a vector bundle to the edges of 
$\G$. In this case there is a vector space $W_e\cong W$ for each edge $e$ as well as each vertex. 
One defines connection isomorphisms $\phi_{ve}=\phi_{ev}^{-1}$ for a vertex $v$ 
and edge $e$ containing that vertex, in such a way that if $e=vv'$ then $\phi_{vv'}=\phi_{ev'}\circ\phi_{ve}$, where $\phi_{vv'}$ is 
the connection on the vertex bundle.

The vertex/edge bundle can be identified with $W^{|\G|+|E|}=W_{\G}\oplus W_E$,
where $W_E$ is the direct sum of the edge vector spaces.

A {\bf $1$-form} (or cochain) is a function on oriented edges which is antisymmetric
under changing orientation. If we fix an orientation for each edge,
a $1$-form is a section of the edge bundle, that is, an element of $W^{|E|}$.
We denote by $\Lambda^1(\G,\Phi)$ the space of $1$-forms
and $\Lambda^0(\G,\Phi)$ the space of $0$-forms, that is, sections of the vertex bundle.

We define a map $d:\Lambda^0(\G,\Phi)\to\Lambda^1(\G,\Phi)$
by $df(e)=\phi_{ye}f(y)-\phi_{xe}f(x)$ where $e=xy$ is an oriented edge from vertex
$x$ to vertex $y$. 
We also define
an operator $d^*:\Lambda^1\to \Lambda^0$ as follows:
$$d^*\omega(v) = \sum_{e=v'v}\phi_{ev}\omega(e)$$ 
where the sum is
over edges containing $v$ and oriented towards $v$. Despite the notation,
this operator $d^*$ is not a standard
adjoint of $d$ unless $\phi_{ev}$ and $\phi_{ve}$ are adjoints
themselves, that is, if parallel transports are unitary operators (see below).

The Laplacian $\Delta$ on $\Lambda^0$ can then be defined as
the operator $\Delta=d^*{\mathcal C}d$ as before:  
\begin{eqnarray*}
d^*{\cal C}df(v)&=& \sum_{e=v'v}c_e\phi_{ev}df(e)\\
&=&\sum_{e=v'v} c_e\phi_{ev}(\phi_{ve}f(v)-
\phi_{v'e}f(v'))\\
&=&\sum_{v'}c_{vv'}(f(v)-\phi_{v'v}f(v'))\\
&=&\Delta f(v).
\end{eqnarray*}

We can see from the example (\ref{3vexample}) above on $K_3$
that $\Delta$ is not necessarily self-adjoint.
However if $\phi_{vv'}$ is unitary: $\phi_{vv'}^{-1}=\phi_{v'v}^*$
then $d^*$ will be the adjoint of $d$ 
for the standard Hermitian inner products
on $W^{|\G|}$ and $W^{|E|}$, and so in this case
$\Delta$ is a Hermitian, positive semidefinite
operator. 
In particular on a line bundle if $|\phi_{vv'}|=1$ for all 
edges $e=vv'$ then
$\Delta$ is Hermitian and positive semidefinite.

\subsection{Minimal networks}

Here we show that any
network on a surface can be ``minimized", that is, made into a minimal network. 

While reconstruction is not possible in general even for minimal
networks (see Theorem \ref{preimage} below), the case of the torus teaches us that 
the map from conductances to the reponse matrix has
potentially interesting preimage. 

A network $\G$ on a surface $\Sigma$ is {\bf minimal} if, 
when lift to the universal cover of $\Sigma$, the network is minimal,
that is, the lifts of the medial strands do not self-intersect and no two 
lifts intersect more than once.

As an example, see Figure \ref{landscape} below.

\begin{thm}\label{minimizable} Every network is topologically equivalent to a minimal
network. 
\end{thm}

\begin{proof} Any surface with a boundary (or closed surface of genus $\ge 2$)
has a hyperbolic metric (that is, a Riemannian metric of constant curvature $-1$)
since its universal cover is the Poincar\'e disk.

Using this metric we can isotope the strands, fixing the endpoints, using a
curve-shortening flow until all strands become hyperbolic geodesics.
Through this isotopy each singularity encountered either reduces the 
number of crossings or is a triple crossing, which corresponds
to a Y-Delta move of the network. In the end 
strands which touch the boundary will have the property
that, when lifted to $\tilde\Sigma$, they meet other strands at most once.
More generally, strands in different homotopy classes meet
each other at most once in the universal cover.

If two or more intersecting strands with no endpoints (that is, which are closed loops
on the surface)
have the same homotopy class
then we must be more careful since their geodesic
representatives will be identical. 
The situation in the universal cover will resemble that of 
Figure \ref{river}, left, that is, we have a packet of strands which
lie near the same geodesic, and other strands cross this packet
transversely and exactly once each.

It remains to prove that in this situation we can replace the self-intersecting packet
with a ``combed" version of itself, where all strands are parallel and non-intersecting
(they will still intersect the transversal strands). For this it suffices
to consider the situation on an annulus in which the packet strands wind around the annulus
forming closed loops, and a certain number of other strands cross this packet transversely.
On the universal cover of the annulus, which is a bi-infinite strip, the situation is as on the left in
Figure \ref{river}. 

We consider, then, a periodic graph on a strip $\R\times[0,1]$
with conductances invariant under $x\mapsto x+1$. 
Suppose we now take $n$ large and 
truncate the strip,
by removing everything far to the left of the origin, say left of $-n$, 
leaving free boundary conditions there, and removing everything to
the right of $+n$, and similarly leaving free boundary conditions there.
What remains is a planar graph on a rectangle $[-n,n]\times[0,1]$, and for this planar graph we can ``comb out"
the packet, as in Figure \ref{river}, from the left to the right, 
so that the strands of the packet
become parallel paths, nonintersecting except just left of the right endpoint,
where they may cross.

We claim that on the resulting combed finite planar graph, the conductances are nearly periodic,
that is, a periodic function plus an error tending to zero as the distance to $\pm n$ increases. 
To see this we use the interpretation of the response matrix entries as
ratios of weighted sums of spanning forests. In \cite{LP2} it is proved that for fixed $i,j$ the truncation
only changes $L_{i,j}$ by amounts tending to zero as $n\to\infty$. It is also shown that
$L_{i,j}\to 0$ as the distance from $i$ to $j$ increases.

Take node $i$ on the top and $j$ on the bottom boundary of the strip.
The numerator of the entry $L_{i,j}$ is the weighted sum of spanning trees wired at the boundary
and with a component connecting $i$ and $j$ (see Theorem \ref{LABC}). We claim that for a typical tree the path from
$i$ to $j$ does not get near $\pm n$ with high probability, that is, this path does not get far from $i$ or $j$.
If the path extended far beyond $j$, then (letting $j+k$ denote the translate of $j$ by $k$)
for large $k$ $L_{i,j+k}$ would be of the same order as $L_{i,j}$, since 
a local change of the long path from $i$ to $j$ would result in a path from $i$ to $j+k$.
This contradicts the fact that $L_{i,j+k}$ tends to zero for large $k$. 

Thus $L_{i,j}$ is a function depending less and
less (as $n\to\infty$) on the conductances far from $i$ and $j$. This proves convergence
of the conductances as $n\to\infty$, and the translation invariance of $L_{i,j}$ implies
the translation invariance of the limiting conductances. 

\end{proof}

\begin{figure}[htpb]
\center{\includegraphics[width=4cm]{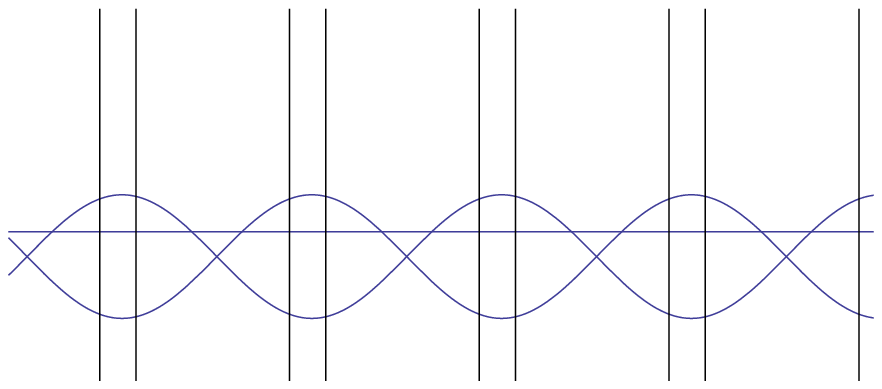}
\includegraphics[width=4cm]{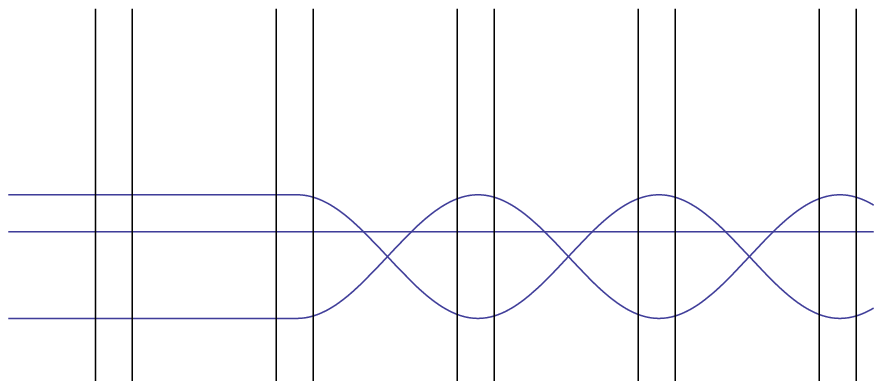}
\includegraphics[width=4cm]{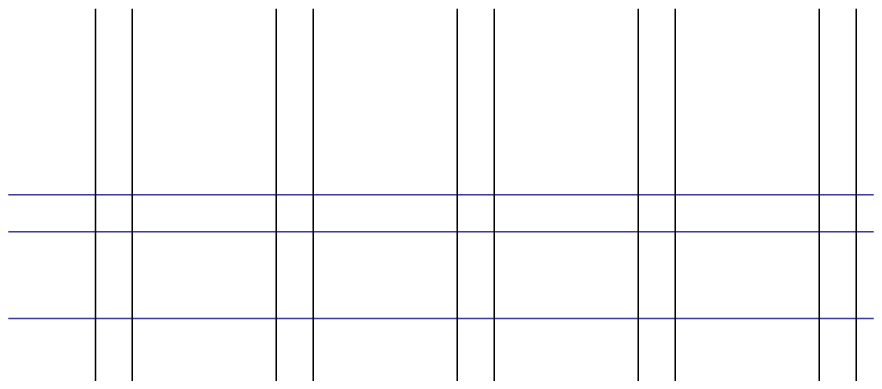}}
\caption{\label{river}Combing a river with three strands.}
\end{figure}

\subsection{Cycle-rooted spanning forests}
Where as the standard Laplacian determinant counts spanning trees, 
in the case of a one-
or two-dimensional vector bundle with connection the bundle
Laplacian determinant counts objects called 
cycle-rooted spanning forests. 
A {\bf cycle-rooted spanning forest} (CRSF) is a subset $S$ of the edges of $\G$,
with the 
property that each connected component
has exactly as many vertices as edges (and so has a unique cycle). See 
Figure \ref{CRSF}.
\begin{figure}[htbp]
\center{\includegraphics[width=3in]{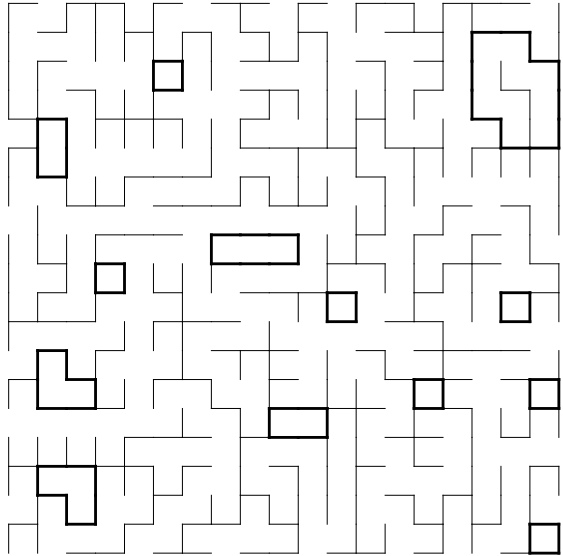}}
\caption{A CRSF on a square grid.\label{CRSF}}
\end{figure}

\begin{thm}[\cite{Forman}]\label{Formanthm}
For a line bundle on a finite graph,
$$\det\Delta=\sum_{\text{CRSFs} T}\prod_{e\in T}c_e\prod_{\text{cycles of }T}(2-w-1/w)$$ 
where the sum is over all CRSFs $T$, the first product is 
over the edges of $T$, and
the second product is over the cycles of $T$, where $w, 1/w$ are the 
monodromies of the two orientations of the cycle. 
\end{thm}

Notice that for a line bundle the monodromies depend only on the cycle
not on the starting point of the cycle.

\begin{thm}[\cite{Kenyon.bundle}]\label{CRSFthm}
For a $\C^2$-bundle on a finite graph with $\SL$ connection, we have
$$\Qdet\Delta=\sum_{\text{CRSFs} T}\prod_{e\in T}c_e\prod_{\text{cycles of }T}(2-\Tr(w))$$ 
where the sum is over all CRSFs $T$, the first product is 
over the edges of $T$, and
the second product is over the cycles of $T$, where $w$ is the 
monodromy of the cycle (starting from some vertex, and in some arbitrary orientation). 
\end{thm}

Notice that the trace of the monodromy is independent of starting point
(since conjugation does not change the trace) and orientation,
since for a matrix $M\in\SL$, we have $\Tr(M)=\Tr(M^{-1})$.

Here the function $\Qdet$ is the {\bf quaternion determinant} of the 
self-dual matrix $\Delta$; this requires some explanation.
A matrix $M$ with entries in $\GL$ is said to be {\bf self-dual} if
$M_{ij}=\tilde M_{ji}$, where 
$$\widetilde{\begin{pmatrix}a&b\\c&d\end{pmatrix}}=
\begin{pmatrix}d&-b\\-c&a\end{pmatrix}.$$
Since $\Delta_{ij}=-c_{ij}\phi_{ij}$ where $\phi_{ij}\in\SL$,
we have $\Delta_{ji}=-c_{ij}\phi^{-1}_{ij}=\tilde\Delta_{ij}$
and so $\Delta$ is self-dual.

We define
$$\Qdet(M) = \sum_{\sigma\in S_n} \text{sgn}(\sigma)
\prod_{\text{cycles}} \frac12\text{tr}(w)$$
where the sum is over the symmetric group,
each permutation $\sigma$ is written as a product of disjoint cycles,
and $\text{tr}(w)$ is the trace of the product of the matrix entries
in that cycle. 
If we group together terms above with the same 
cycles---up to the order of traversal of each cycle---then the contribution
from each of these terms
is identical: reversing the orientation of a cycle does not change
its trace. 
So we can write
$$\Qdet(M)=\sum_{\text{cycle decomps}} (-1)^{c+n}\prod_{i=1}^c \widehat{\text{tr}}(w_i),$$
where the sum is over cycle decompositions of the indices (not taking
into account the orientation of the cycles), $c$ is the number of cycles,
and $w_i$ is the monodromy (in one direction or the other)
of each cycle. Here $\widehat{\text{tr}}$ is equal to the trace for 
cycles of length at least $3$; cycles of length $1$ or $2$
are their own reversals so we define
$\widehat{\text{tr}}(w)= \frac12\text{tr}(w)$
for these cycles.

As an example, let $A=aI, C=cI$ and $B=\left(\begin{matrix}
b_1&b_2\\b_3&b_4\end{matrix}\right)$. Then
$$\text{Qdet}\left(\begin{matrix}A&B\\\tilde B&C\end{matrix}\right)=
\widehat{\text{tr}}(A)\widehat{\text{tr}}(C)-\widehat{\text{tr}}(B\tilde B) = 
ac-(b_1b_4-b_2b_3).$$

Note that if $M$ is a self-dual $n\times n$ matrix
then $ZM$, considered as a $2n\times 2n$ matrix is antisymmetric,
where $Z$ is the matrix with diagonal blocks 
$\left(\begin{matrix}0&-1\\1&0\end{matrix}\right)$ and zeros elsewhere.
The following theorem allows us to compute 
$Q$-determinants explicitly.

\begin{theorem}[\cite{Dyson}]
Let $M$ be an $n\times n$ self-dual matrix with entries in $GL_2(\C)$
and $M'$ the associated $2n\times 2n$ matrix, obtained by replacing
each entry with the $2\times2$ block of its entries.
Then $\Qdet(M)=\mathrm{Pf}(Z M')$, the Pfaffian of
the antisymmetric matrix $ZM'$.
\end{theorem}

Note that the matrix $\Delta'$ (obtained by replacing entries in $\Delta$ with
the $2\times2$ block of their coefficients) is just the 
matrix $\Delta$ acting on the total space of the bundle $W_{\G}$. 
So up to a sign we can write 
$$\Qdet\Delta=(\det\Delta)^{1/2}.$$

\subsubsection{Unitary connections and measures}

In the case of a line bundle, 
if $\prod_{\text{cycles}}2-w-1/w\ge 0$ for every CRSF,
we can define, following Theorem \ref{Formanthm},
a probability measure on CRSFs 
in which a CRSF has probability proportional to the product of 
its edge weights times the factor $\prod_{\text{cycles}}2-w-1/w$.

For example in the case that $|\phi_e|=1$ for every edge,
that is, $\phi$ is a {\bf unitary connection},
then $\Delta$ is a Hermitian, positive semi-definite matrix
(and positive definite for generic unitary connections).
In this case $|w|=1$ for all monodromies and so $2-w-1/w\ge 0$,
and equal to zero only for cycles with trivial monodromy.

Similarly in the case of a two-dimensional bundle if $\phi_e\in\SU$
for all edges then $\Delta$ is unitary and $2-\Tr w\ge 0$ for all monodromies.

In both cases we call $\mu=\mu_{\Phi}$ the corresponding probability measure.

Examples of natural settings of these measures are when the graphs
are embedded on surfaces with geometric structures, such as
a Riemannian metric (in which case the Levi-Civita connection defines
parallel transport of vectors in the tangent bundle across an edge and provides
a $U(1)$-connection) or an $\SU$ structure on the surface and parallel transport restricted to the graph gives a $\SU$ connection.

In the case of a flat connection, the monodromy of a contractible cycle is trivial,
so that the probability measure is supported on CRSFs whose 
cycles are all topologically nontrivial. We call such CRSFs {\bf essential CRSFs}.

\subsection{Cycle-rooted groves}

For the bundle Laplacian, the natural objects replacing groves are
cycle-rooted groves. Given a graph $\G$ embedded on a surface
with nodes $\No$, a {\bf cycle-rooted grove} (CRG) is a collection of edges
with the property that each component is either
\begin{enumerate}
\item A tree containing one or more nodes
\item a cycle-rooted tree containing no nodes.
\end{enumerate}

Theorems \ref{Formanthm} and \ref{CRSFthm} have an extension to the 
case of a graph with boundary $B\subset V$; the results are that the
bundle Laplacian determinants (for the Laplacian with boundary) are weighted
sums of cycle rooted groves; the weights of components with cycles
are as before, and the weights of tree components are just the 
product of their edge weights (there is no ``monodromy" contribution for tree components).
See \cite{KW4}.

\section{Annular networks}

\subsection{Laplacian determinant}
Let $\G$ be a network with flat $\C^*$-connection on an annulus, and
with no nodes.
Let $z$ be the monodromy around a loop generating the homotopy group of the annulus.
Then any simple closed loop on $\G$ will have monodromy $1$ if contractible,
$z$ or $z^{-1}$ if noncontractible. 
 
As a consequence by Theorem \ref{Formanthm}
the Laplacian determinant is 
\be\label{lapdetann}
\det\Delta = \sum_{k=1}^\infty C_k(2-z-\frac1z)^k,\ee
where $C_k$ is the weighted sum of CRSFs having $k$ components.

Thus $P(z)=\det\Delta$ is a Laurent polynomial in $z$ from which one can
extract the information about the number of loops in a random sample of a CRSF
on $\G$. For example substituting $x=2-z-1/z$ we have that $P(x)/P(1)$ is the
probability generating function for the number of components.

Independently of the conductances, $P(z)$ contains
topological information about the graph $\G$, for example
the highest power of $z$ is the maximum 
number of disjoint loops which can be drawn
on $\G$, each winding around the annulus. This is because any such set of loops
can be completed to a grove by adding some edges. 

There is a simple characterization of the roots of $P$.
By (\ref{lapdetann}) the roots come in reciprocal pairs
one of which is a double root at $z=1$. 

\begin{thm}\label{distinct}
$P$ is reciprocal and the roots of $P$ are real, positive, and distinct except for a double root at $z=1$.
\end{thm}

\begin{proof}
Our first goal is to minimize the graph so that it is a string of loops as illustrated in Figure \ref{loopline}, 
each loop winding around the annulus.

As in the proof of Theorem \ref{minimizable} above, we use the curve-shortening flow
on the medial strands until we have only a packet of strands winding around the core of the annulus
as in Figure \ref{river}, left (but without the vertical strands). 
In the resulting network, choose vertices $v_1,v_2$, one adjacent to each boundary
and temporarily consider these to be nodes. There is one strand wrapping outside $v_1$
and one wrapping outside $v_2$. Break these strands so they have two stubs near each of $v_1,v_2$.
Now continue to minimize the medial strands as in the proof of Theorem \ref{minimizable}. In the minimal
network, the stubs from $v_1$ must cross to the stubs of $v_2$ (rather than to each other),
since otherwise the network would be disconnected. We can push all intersections of the finite strands
(that is, those connecting stubs) to one side of the intersections of the stub/river strands, as described in
Section \ref{lss} and Figure \ref{landscape} below. 
The network and medial graph now resemble that in Figure \ref{house}.
The graph can now easily be converted into a string of loops using the sequence of 
operations illustrated in Figure \ref{housetowall}.

\begin{figure}[htbp]
\center{\includegraphics[width=10cm]{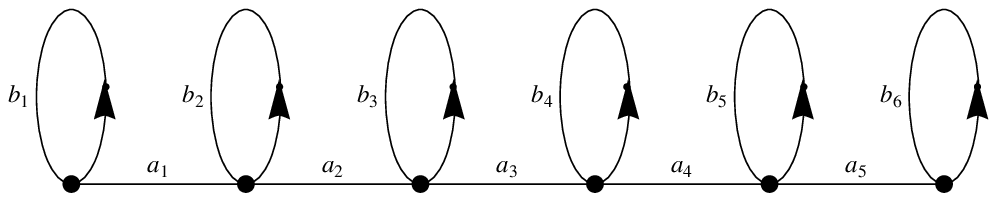}}
\caption{\label{loopline}Every graph on an annulus can be minimized to 
a string of loops.}
\end{figure}
\begin{figure}[htbp]
\center{\includegraphics[width=10cm]{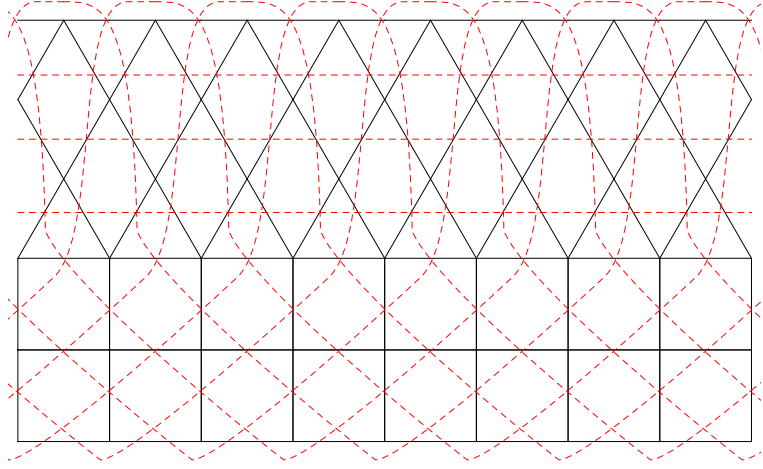}}
\caption{\label{house}Intermediate stage in minimization of network on an annulus. (Medial strands dashed/red).}
\end{figure}
\begin{figure}[htbp]
\center{\includegraphics[width=5cm]{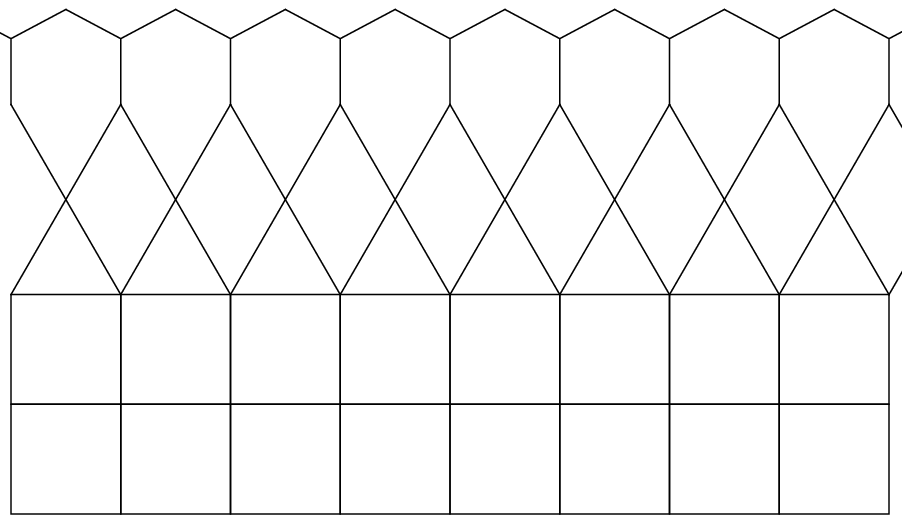}\hspace{0.5cm}\includegraphics[width=5cm]{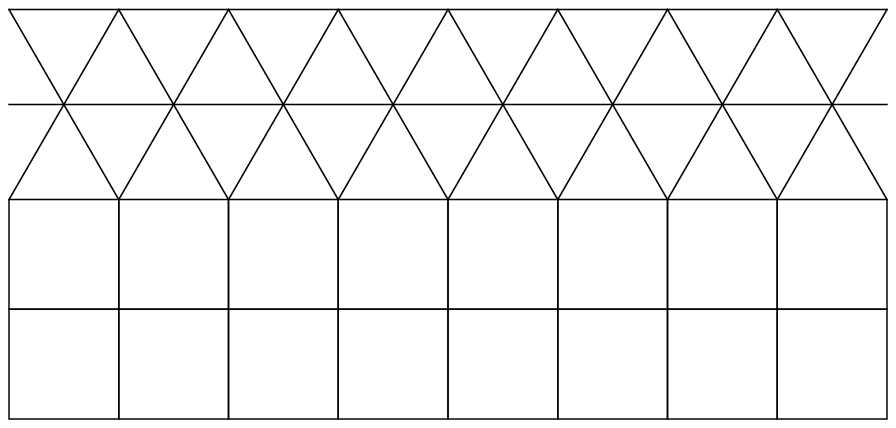}}
\center{
\includegraphics[width=5cm]{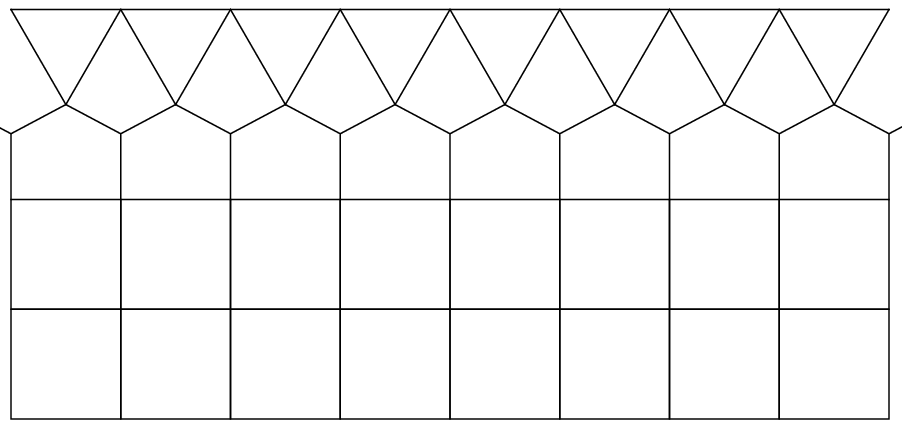}\hspace{0.5cm}\includegraphics[width=5cm]{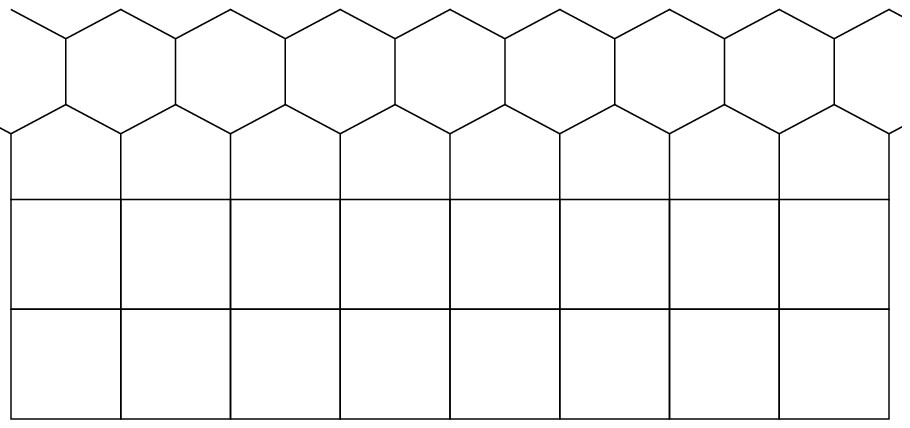}}
\center{\includegraphics[width=5cm]{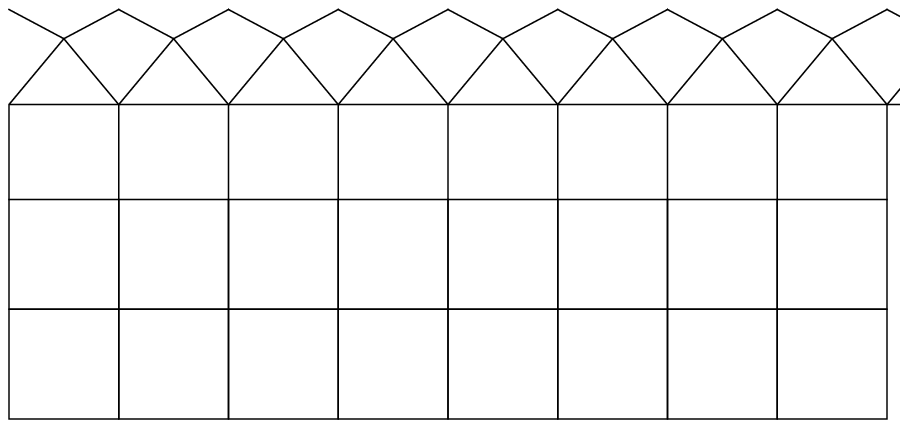}\hspace{0.5cm}\includegraphics[width=5cm]{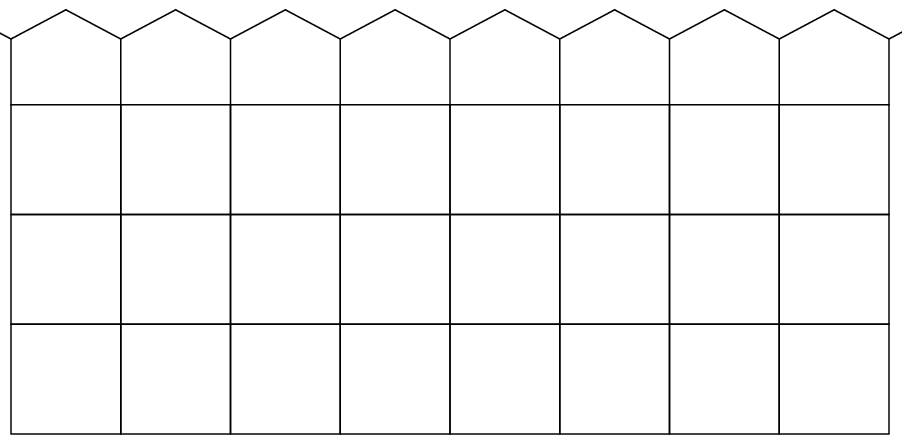}}
\center{\includegraphics[width=5cm]{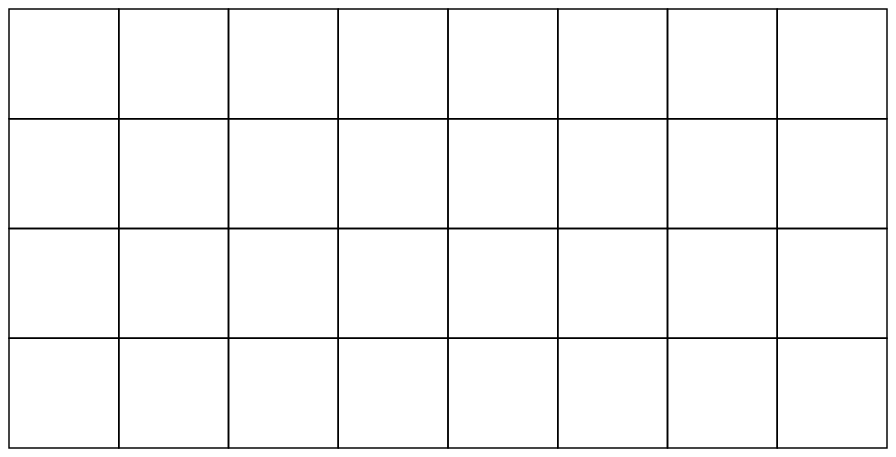}}
\caption{\label{housetowall}Passing from the intermediate stage to the string of loops: (1) From the network
in Figure \protect\ref{house}, use Delta-Y moves on the upper row of triangles. (1)-(2) Flatten the degree-$2$ vertices on the upper row and use Y-Deltas on the second row of vertices.
(2)-(3) Use Delta-Ys on the row of triangles just above the squares. (3)-(4) Use Delta-Ys on the upper row of triangles. (4)-(5) Y-Delta. (5)-(6) Flatten and then Delta-Y. (6)-(7) Flatten.}
\end{figure}
Let $a_1,\dots,a_{n-1}$ be the conductances on the edges of the string and
$b_1,\dots,b_n$ the conductances on the loops.
Then the Laplacian is a tridiagonal matrix (using $X=2-z-1/z$)
$$\Delta=\begin{pmatrix}a_1+b_1X&-a_1&0&\\
-a_1&a_1+a_2+b_2X&-a_2&0&\\
0&-a_2&\ddots\\
&&&a_{n-2}+a_{n-1}+b_{n-1}X&-a_{n-1}\\
&&&-a_{n-1}&a_{n-1}+b_nX.
\end{pmatrix}
$$
We pre- and post-multiply by the diagonal matrix $U$
whose diagonal entries are $b_i^{-1/2}$
to get
$$U\Delta U=\begin{pmatrix}\frac{a_1}{b_1}+X&\frac{-a_1}{\sqrt{b_1b_2}}&0&\\
\frac{-a_1}{\sqrt{b_1b_2}}&\frac{a_1}{b_2}+\frac{a_2}{b_2}+X&\frac{-a_2}{\sqrt{b_2b_3}}&0&\dots\\
0&&\ddots\\
&&&\frac{a_{n-2}}{b_{n-1}}+\frac{a_{n-1}}{b_{n-1}}+X&\frac{-a_{n-1}}{\sqrt{b_{n-1}b_n}}\\
&&&\frac{-a_{n-1}}{\sqrt{b_{n-1}b_n}}&\frac{a_{n-1}}{b_n}+X.
\end{pmatrix}
$$
Thus we see that for each root $z$ of $P$, $-2+z+\frac1z$ is an eigenvalue of 
$$M=\begin{pmatrix}\frac{a_1}{b_1}&\frac{-a_1}{\sqrt{b_1b_2}}&0&\\
\frac{-a_1}{\sqrt{b_1b_2}}&\frac{a_1}{b_2}+\frac{a_2}{b_2}&\frac{-a_2}{\sqrt{b_2b_3}}&0&\dots\\
0&&\ddots\\
&&&\frac{a_{n-2}}{b_{n-1}}+\frac{a_{n-1}}{b_{n-1}}&\frac{-a_{n-1}}{\sqrt{b_{n-1}b_n}}\\
&&&\frac{-a_{n-1}}{\sqrt{b_{n-1}b_n}}&\frac{a_{n-1}}{b_n}.
\end{pmatrix}.
$$
This matrix $M$ is symmetric and so has real eigenvalues, which are
positive by (\ref{lapdetann}). Symmetry also implies that $M$ has
orthogonal eigenvectors. Thus if $-X$ is a multiple
eigenvalue of $M$ 
it has an eigenspace $V_{-X}$ of dimension $\ge 2$. 
This means there is a nonzero eigenvector of $M$ with eigenvalue $-X$
whose first coordinate is $0$. But the coordinates of an eigenvector
$(x_1,x_2,\dots,x_n)$ satisfy a length-three linear recurrence:
$$\lambda x_1 = \frac{a_1}{b_1}x_1 - \frac{a_1}{\sqrt{b_1b_{2}}}x_{2}$$
and for $i>1$
$$\lambda x_i = -\frac{a_{i-1}}{\sqrt{b_{i-1}b_i}}x_{i-1}+(\frac{a_{i-1}}{b_i}+\frac{a_i}{b_i})x_i - \frac{a_{i}}{\sqrt{b_{i}b_{i+1}}}x_{i+1}.$$
Starting from $x_1=0$ this  implies (using that all $a_i\ne 0$) 
that $x_i=0$ for all $i$.
This is a contradiction. We conclude that the eigenvalues of $M$,
and therefore the roots of $P$,
are distinct.
\end{proof}

\old{\begin{cor}
For arbitrary strictly positive $c_1,\dots,c_n,d_1,\dots,d_{n-1}$ the matrix
$$\begin{pmatrix}c_1&-d_1&0&\dots\\0&c_2&-d_2&0\\
&0&\ddots&-d_{n-1}\\
&&0&c_n
\end{pmatrix}$$ (with zeros off the diagonal and superdiagonal)
has distinct singular values.
\end{cor}

\begin{proof} The matrix $M$ of (\ref{M}) can be written $M=U^tU$ where $U$ has the form
of the statement with $c_i=\sqrt{a_i/b_i}$ for $i=1,\dots,n-1$, $c_n=0$,
and $d_i=\sqrt{a_i}{b_{i+1}}$. If the $a_i,b_i$ are arbitrary then so are
the $c_i,d_i$. 
\end{proof}
}

\subsection{Cylinder example}
Let us compute, for a rectangular cylinder, $\det\Delta$ for the flat line bundle with
monodromy $z$. This will allow us to compute the corresponding
distribution of cycles in a uniform random essential CRSF.
Let $H_{m,n}$ be the square grid network on a cylinder, obtained
from the $n\times m$ square grid by adding edges from $(n,i)$ to $(0,i)$
for each $i\in[1,m]$.
$H_{m,n}$ is a ``product" of an $m$-vertex linear network $\G_m$
and a circular network $\Z_n$ of length $n$.
We put a flat line bundle structure on $H_{m,n}$ by putting 
parallel transport $z$ on the edges $(n,i)(0,i)$ and $1$ on all other edges.

The eigenvalues of the Laplacian $\Delta_{\G_m}$ on the linear network
are $2-2\cos\frac{k\pi}{m}$ for $k=1,2,\dots,m$
(the corresponding eigenvectors are $f_k(x) = \cos\frac{\pi k(x+1/2)}{m}$ for $x=0,1,2,\dots,m-1$).

The eigenvalues of the line-bundle Laplacian $\Delta_{\Z_n}$ 
on $\Z_n$ with monodromy $z$
are $2-\zeta-1/\zeta$ where $\zeta$ is an $n$th root of $z$.
(The corresponding eigenvectors are $g_k(x)=\zeta^{x}$.)
The eigenvectors on $H_{m,n}$ are the products $f_kg_\ell$
for $(k,\ell)\in[1,m]\times[0,n-1]$.

We then have
$$\det\Delta_{\Phi}=\prod_{\zeta^n=z}\prod_{k=1}^m4-2\cos\frac{k\pi}{m}-\zeta-\frac1{\zeta}.$$
\newcommand{\Ch}{\text{Ch}}
Use the identity $$\prod_{\zeta^n=z} R-\zeta-\frac1{\zeta}=\Ch_n(R)$$
where $\Ch_n$ is defined by
$\Ch_n(\alpha+\frac1{\alpha})=\alpha^n+\alpha^{-n}$ (a variant of the Chebychev polynomial). 

This leads to 
$$\det\Delta=(2-z-z^{-1})\prod_{k=1}^m(\Ch_n(4-2\cos\frac{k\pi}{m})-z-\frac1z)$$
$$=w\prod_{k=1}^m w+\Ch_n(4-2\cos\frac{k\pi}{m})-2$$
where $w=2-z-\frac1z$.
If we let $Q(w)$ be this polynomial then $Q(w)/Q(1)$ is the probability generating
function of the number of cycles in a uniform random essential CRSF.
 
It is interesting to consider what happens to this distribution for a 
large annulus, when $m,n\to\infty$ with $m/n$ converging to a
fixed quantity $\tau$.
For large $n$, $\Ch_n(4-2\cos\frac{k\pi}{m})$ is large unless $k$ is near $0$.
Thus only values of $k$ near $0$ affect the limiting distribution.
We have $4-2\cos\frac{k\pi}{m}=2+\frac{\pi^2k^2}{m^2}+O(\frac{k}m)^4 = \alpha_k+1/\alpha_k$ where
$\alpha_k=1+\frac{\pi k}{m}+O(\frac{k}m)^2$.
Thus in the limit $m,n\to\infty$ with $m/n\to\tau$ we have
$\Ch_n(4-2\cos\frac{k\pi}{m})=2\cosh\frac{\pi k}{\tau}+o(1)$.
The limit probability generating function for the number of cycles is then
$$Q(w)/Q(1)=w\prod_{k=1}^\infty\left(\frac{w+2\cosh\frac{\pi k}{\tau}-2}{2\cosh\frac{\pi k}{\tau}-1}\right).$$
See Figure \ref{annuluscycles}.
\begin{figure}[htbp]
\center{\includegraphics[width=10cm]{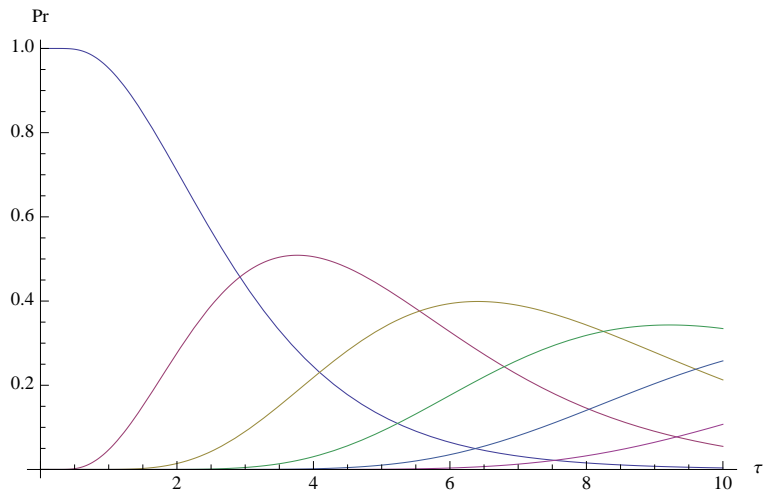}}
\caption{\label{annuluscycles}The probabilities of $1,2,\dots,6$ cycles
in a uniform random essential CRSF on a large annulus, as a function of the
modulus $\tau$.}
\end{figure}

\subsection{The response matrix}

Lam and Pylyavskyy \cite{LP2} studied the response matrices of networks on annuli.
Their point of view was to consider the lift $\tilde\G$ of the network $\G$
to the universal cover of the annulus. 
There one can define the response matrix as the limit of response matrices on larger
and larger portions of the graph; it is not hard to show that this limit exists.
Since the universal cover is planar one can recover some of the results from the 
planar case like nonnegativity of the non-interlaced minors (non-interlacedness
makes sense since the lift $\tilde\G$ is planar). 

We will take a different approach here which uses the bundle Laplacian
for a flat connection.
Let $\G$ be a network drawn on an annulus $\Sigma$, 
with nodes $\No$ which are a 
subset of the vertices adjacent to the two boundaries of $\Sigma$. 
Let $\Phi$ be a flat $\C^*$-connection with monodromy $z$ around a generator
of $\pi_1(\Sigma)$. 
Let $\Delta$ be the associated bundle Laplacian.
Let $L$ be the response matrix, defined as before as
$$L=-A+BC^{-1}B^*$$
where $\Delta=\begin{pmatrix}A&B\\B^*&C\end{pmatrix}.$
Now entries in $L$ are rational functions of $z$ (with coefficients which depend on the 
conductances).

What can be said about EIT on annular networks?

For circular planar networks it was very useful to have combinatorial interpretations (in terms of groves)
of the entries and minors of the response matrix.
We have similar interpretations in the present case.
The following theorem of \cite{KW4} holds for any network with line bundle
(not just networks on an annulus). Compare Theorem \ref{LABC} above.

\begin{theorem}[\cite{KW4}]\label{LABCbundle}
  Let $Q,R,S,T$ be a partition of $\No$ and $|R|=|S|$.
  Then $\det L_{R\cup T}^{S\cup T}$ is the ratio of two terms: the denominator
  is the weighted sum of CRSFs; the numerator is a signed weighted sum
  of cycle-rooted groves of $\G_T$, the graph $\G$ in which all nodes in $T$ are
  considered internal, the nodes in $Q$ are in singleton parts,
  and in which nodes in $R$ are paired with nodes
  in $S$, with the sign being the sign of the pairing permutation,
  times the parallel transports from $S$ to $R$:
$$
 \det L_{R,T}^{S,T} = \sum_{\text{\textrm{permutations} $\rho$}}(-1)^\rho\frac{\Zv\big[{}_{r_1}^{s_{\rho(1)}}|\cdots|_{r_k}^{s_{\rho(k)}}|q_1|\cdots|q_\ell\big]}{\Zv[1|2|\cdots|n]}
$$
\end{theorem}

Here the numerator
$\Zv\big[{}_{r_1}^{s_{\rho(1)}}|\cdots|_{r_k}^{s_{\rho(k)}}|q_1|\cdots|q_\ell\big]$
is the weighted sum of CRG's in which there are tree components connecting
$r_i$ to $s_{\rho(i)}$ for each $i$; the weight is the product of the edge
conductances, times the product over all cycles $\gamma$ of $2-w(\gamma)-1/w(\gamma)$ (where $w(\gamma)$ is the monodromy of the cycle $\gamma$)
and times the product of the
parallel transports from the $r_i$
to the $s_{\rho(i)}$ along the edges of the trees. We use the notation
$\Zv$ as opposed to $Z$ to remind us that we are dealing with
the bundle Laplacian.

On an annulus with flat connection, it is convenient to have the connection
supported on a ``zipper", that is, the set of faces crossing a shortest 
path in the dual graph
from one boundary component to the other, as in Figure \ref{zipper}.
In such a case the monodromy along a path from a node to another node is 
$1, z$ or $z^{-1}$ if the nodes are on the same boundary; it is a power of
$z$ if the nodes are on different boundaries, and it is $z^{\pm 1}$ for a 
topologically nontrivial cycle.
\begin{figure}[htbp]
\center{\includegraphics[width=6cm]{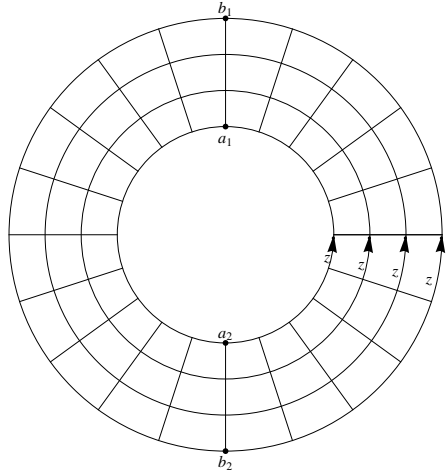}}
\caption{\label{zipper}The parallel transport on an annulus can be supported
on a zipper.}
\end{figure}

If particular if $i$ and $j$ are nodes on opposite boundaries then
$$L_{ij}\Zvunc=\sum_{k\in\Z} C_kz^k,$$
where $C_k$ is the weighted sum of CRGs with a component
tree connecting $i$ to $j$, no cycles (a cycle is precluded
by the existence of a path from $i$ to $j$), and where $k$ is the signed number
of times the path from $i$ to $j$ crosses the zipper. 

Thus $L_{ij}\Zvunc$ is a Laurent polynomial of $z$ with nonnegative coefficients.

Here is another example. Suppose $A=\{a_1,a_2\}$ are nodes on one boundary
and $B=\{b_1,b_2\}$ are nodes on the other. 
Then by the Theorem $L_A^B\Zvunc$ contains terms in which $a_1$ is paired with $b_1$
(and $a_2$ with $b_2$) and, with an opposite sign terms in which
$a_1$ is paired with $b_2$ and $a_2$ with $b_1$. However these two types of
terms differ by an power of $z$. Suppose for example that the zipper is as shown in Figure \ref{zipper}. Then any term ${}_{a_1}^{b_1}|{}_{a_2}^{b_2}$ has an even power of $z$ and ${}_{a_1}^{b_2}|{}_{a_2}^{b_1}$ has an odd power of $z$. Thus 
$L_A^B\Zvunc$ is a Laurent polynomial in $z$ with coefficients of alternating sign.

\subsection{Minimality and reconstruction}

\subsubsection{Minimal networks on the annulus: landscapes}\label{lss}

Let $\G$ be a minimal network on an annulus $\Sigma$ with $n_1$
nodes on one boundary and $n_2$ on the other.
As discussed, let $\tilde\G$ be the lift of $\G$ to the strip,
the universal cover of the annulus.
The medial strands come in four types:
\begin{enumerate}
\item they connect stubs on the lower boundary of the strip (the {\bf rocks});
\item they connect stubs on the upper boundary of the strip (the {\bf clouds});
\item they connect stubs across the boundaries (the {\bf trees});
\item they form bi-infinite paths along the strip (the {\bf river}).
\end{enumerate}

We can perform Y-Delta transformations so that the strands of the first, second
and fourth type do not intersect (that is, no strand of the first type intersects
a strand of the second or fourth type, although it may intersect
other strands of the first type, etc.) As long as there is at least one tree
strand, then we can ``untwist" the river strands so that they are all parallel
and noncrossing.

The picture is then like a landscape; the rocks at the bottom,
interspersed with trees which grow up to mingle with the clouds,
and the river runs through the trees but does not touch the clouds or rocks.
See Figure \ref{landscape} and \ref{landscape2}.

\begin{figure}[htpb]
\center{\includegraphics[width=5.5in]{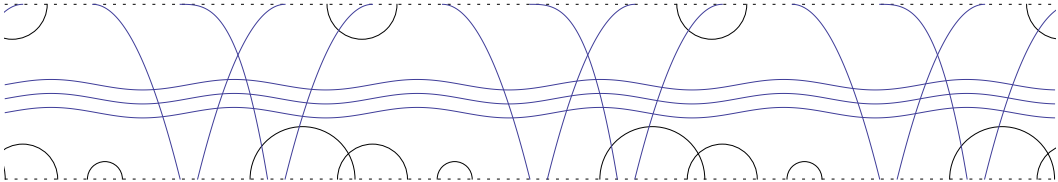}}
\caption{\label{landscape}A landscape: the general form of a minimal network on an annulus
(we show only the medial graph, and on the universal cover of the annulus).}
\end{figure} 

\begin{figure}[htpb]
\center{\includegraphics[width=3in]{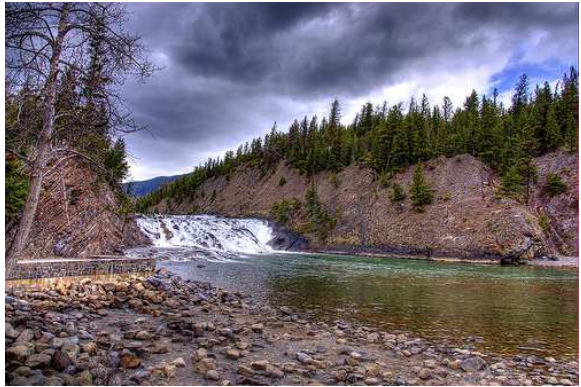}}
\caption{\label{landscape2}Bow falls, Banff.}
\end{figure} 

\subsubsection{Reconstruction}

The reconstruction of the conductances on a minimal annular network as a function
of the response matrix is open in general. Lam and Pylyavskyy \cite{LP2}
worked out the reconstruction for the ``grid" network of Figure
\ref{gridannulus}, whose medial graph is a grid.
\begin{figure}[htpb]
\center{\includegraphics[width=6cm]{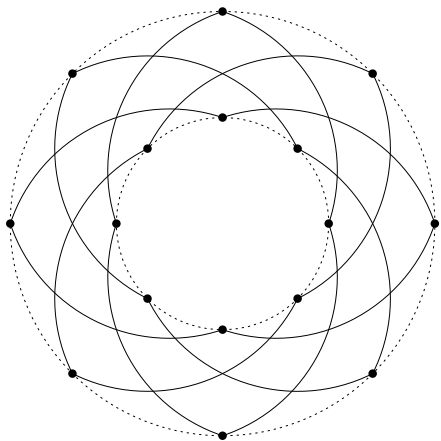}}
\caption{\label{gridannulus}The ``grid" network on an annulus.}
\end{figure}

The new feature of this reconstruction problem is that the solution is not unique.
If the medial graph of $\G$ has $k$ topologically nontrivial cycles, then there
are generically $k!$ solutions to the reconstruction. The idea is that
given a set of conductances on $\G$ one can permute the parallel
strands of the medial graph and get another solution. 
How does one permute two strands? This follows from 
the discussion of combing in Theorem
\ref{minimizable} above. 
See Figure \ref{river}. One introduces a pair of edges of opposite
conductance $c,-c$; this has no effect on the current flow or response matrix.
Now perform Y-Delta moves around the annulus to exchange the adjacent
strands. There is a unique choice of $c$ for which the added edge, when it comes
back around, will again have conductance $c$; it then can be removed along with the
parallel edge of conductance $-c$.

In \cite{LP2} it is conjectured that reconstruction is possible on all minimal annular networks.

\section{Periodic networks and networks on the torus}

Let $\G$ be a network on a torus. By this we mean a network $\G$,
with no boundary, embedded on a torus in such a way that every complementary
component is contractible.

For concreteness we suppose the torus is $\T^2=\R^2/\Z^2$
(but the network can be arbitrary).
The lifted network $\tilde\G$ on the universal cover $\R^2$ of $\T^2$
is a periodic planar network.

This case differs from the previous cases of circular planar and annular networks
in that the underlying surface has no boundary, so we do not introduce
any nodes.
Remarkably we can still have a complete theory about reducibility, minimal networks,
and the reconstruction problem. 

\subsection{The spectral curve of $\Delta$}

Let $\Phi$ be a flat line bundle on $\G$ with monodromy $z_1,z_2$
around the standard generators of the homotopy group of $\T^2$. 
Let $\Delta=\Delta_\Phi$ the associated Laplacian.
See for example Figure \ref{2X2torus}.
In this example the associated Laplacian
determinant $P(z_1,z_2)=\det\Delta$ is (with the vertices in the order indicated)
$$P(z_1,z_2)=\det\left(\begin{array}{cccc}6&-3-z_1^{-1}&-1-z_1^{-1}&0\\-3-z_1&6&0&-1-z_2^{-1}\\-1-z_2&0&4&-1-z_1^{-1}\\0&-1-z_2&-1-z_1&4\end{array}\right)=$$
\be\label{Pexample}
=
3 z_1^2+\frac{3}{z_1^2}-4 z_1 z_2-\frac{4z_1}{z_2}-\frac{4 z_2}{z_1}-\frac{4}{z_1z_2}-76z_1-\frac{76}{z_1}+z_2^2+\frac{1}{z_2^2}-52z_2-\frac{52}{z_2}+264.\ee

\begin{figure}
\center{\includegraphics[width=6cm]{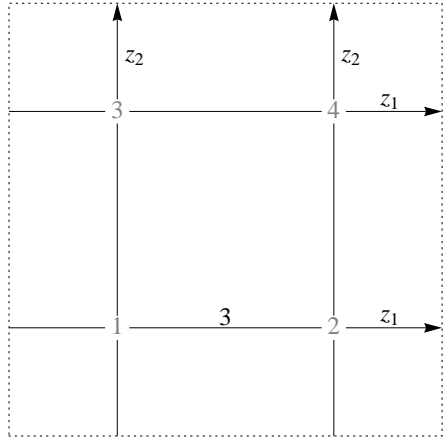}}
\caption{\label{2X2torus}The $2\times 2$ grid on a torus; in this example
all edges have conductance $1$ except for the one indicated
which has conductance $3$.} \end{figure}

In general $P$ is a (Laurent) polynomial with
coefficients which are polynomials in the conductances.
It is called the {\bf characteristic polynomial} of $\Delta$.
The curve $\{(z_1,z_2)\in\C^2~|~P(z_1,z_2)=0\}$ is the {\bf spectral curve} of $\Delta$.
Note that 
$P$ is symmetric: $P(z_1,z_2)=P(z_1^{-1},z_2^{-1})$.
This is because we can write $\Delta=d^*{\cal C}d$ where $d^*$ is the Hermitian
transpose of $d$ when $|z_1|=|z_2|=1$.

The {\bf Newton polygon} of the network $\G$ on the torus
is the Newton polygon of $P$,
that is, the convex hull of 
$$\{(i,j)\in\Z^2~|~z_1^{i_1}z_2^{j_2} \text{ is a coefficient of $P$}\}.$$
It is a centrally symmetric polygon centered at $(0,0)$.

By Theorem \ref{Formanthm} we can write
\begin{eqnarray}
P(z_1,z_2)&=&\sum_{\text{CRSFs}}\left(\prod_ec_e\right)(2-w-1/w)^k,\nonumber\\
&=&\sum_{(r,s)\ne(0,0)}C_{r,s}(2-z_1^iz_2^j-z_1^{-i}z_2^{-j})^k,\label{crsfform}
\end{eqnarray}
where in the first line we used the fact that in a CRSF on a torus all cycles
have the same homology class, and in the second line
we sum over all homology classes $(r,s)\in\Z^2$ 
(we need only
sum over one of each pair $(r,s),(-r,-s)$) where $(r,s)=(ik,jk)$ with $i,j$
being relatively prime. The coefficient $C_{r,s}$ is the weighted sum of CRSFs
with $k$ cycles of homology class $(i,j)$, that is (if we orient
the cycles in the same direction) of total homology class $(r,s)$.

This shows that the $z_1^rz_2^s$ coefficient of $P(z_1,z_2)$ 
is only due to the CRSFs with cycles in direction
$r/s$. Thus the Newton polygon
$N$ tells us which homology classes of CRSFs are possible
for $\G$; equivalently, the maximal number of vertex-disjoint cycles in 
$\G$ which have homology class $(i,j)$ is determined by the integer
point in $N$ which is farthest from the origin in direction $(i,j)$. 

In the above example (\ref{Pexample}), $P$ can be written in ``Newton polygon form"
$$P=\begin{matrix}&&1&&\\&-4&-52&-4&\\3&-76&264&-76&3\\
&-4&-52&-4&\\&&1&&
\end{matrix}.
$$
From this, one can write $P$ in a unique way in form (\ref{crsfform}):
\begin{multline} P(z_1,z_2)=3(2-z_1-\frac1{z_1})^2+64(2-z_1-\frac1{z_1})+(2-z_2-\frac1{z_2})^2+
48(2-z_2-\frac1{z_2})+\\
+4(2-z_1z_2-\frac{1}{z_1z_2})+4(2-\frac{z_1}{z_2}-\frac{z_2}{z_1}).
\end{multline}

Thus there are, for example, exactly $48$ CRSFs with a single cycle which has homology class $(0,\pm1)$.

\subsection{UST on the plane}

Let us say briefly a few words about the UST on the plane.

Let $\tilde\G$ be the periodic planar graph which is the lift
of $\G$ to $\R^2$. 
Pemantle \cite{Pemantle} showed that there is a UST measure 
on spanning trees of $\tilde\G$ which is the limit of the UST
measure on larger and larger tori 
$\tilde\G_n \stackrel{\text{def}}{=} 
\tilde\G/n\Z^2$. 

The {\bf Free Energy} of this measure is the limit
$$F=\lim_{n\to\infty}\frac1{n^2}\log Z_n,$$
where $Z_n$ is the partition sum for spanning trees on $\tilde\G_n$.

There is a nice formula for $F$
$$F=\frac1{(2\pi i)^2}\int_{|z_1|=1}\int_{|z_2|=1}\log P(z_1,z_2)\frac{dz_1}{z_1}\frac{dz_2}{z_2}$$
which arises from the exact formula
$$Z_n(z_1,z_2)=
\prod_{\zeta_1^n=z_1}\prod_{\zeta_2^n=z_2}P(\zeta_1,\zeta_2)$$
(obtained by Fourier analysis of $\tilde\G_n$) by taking logs and replacing the Riemann sum with an integral. 

One can analyze the limiting UST measure using the fact
that edge process is determinantal with kernel $T$
given by the {\bf Transfer Current} matrix:
$T(e,e')$ is the current across $e'$ when one unit of current
enters the graph at $e_+$ and exits at $e_-$.
In terms of the Greens function we can write
$$T(e,e') = G(e_+,e'_+)-G(e_+,e'_-)-G(e_-,e'_+)+G(e_-,e'_-).$$
The Green's function $G(v,v')$ on $\tilde\G$ has a formula involving 
the Fourier coefficients of $1/P(z_1,z_2)$: more precisely,
an integral over the unit torus $|z_1|=|z_2|=1$ of an expression
$z_1^xz_2^yQ(z_1,z_2)/P(z_1,z_2)$ where $(x,y)$ is the translation
from $v$ to $v'$ and $Q$ is a polynomial depending on where $v,v'$
sit in their respective fundamental domains. 
In particular the transfer current has an exact
asymptotic expression. 

\subsection{Medial strands and minimality}

Recall that two networks on $\T^2$ are {\bf topologically equivalent} if one can be obtained from the other (disregarding conductances) by electrical transformations. 
We say that they are {\bf electrically equivalent}
if their characteristic polynomials are the same (compare this definition with
the one for networks with boundary in section \ref{elecequiv}).

\begin{thm}\label{topequivtorus}
Two networks on a torus 
are topologically equivalent if and only if their Newton polygons
are equal.
\end{thm}

\begin{proof}
It is not hard to show that electrical transformations do not change $N$.
The simplest reason is because there is a bijection 
on the CRSFs of the ``before" and ``after"
networks, and the set of CRSFs determines $N$.

So it suffices to show that $N$ determines a minimal network up to 
electrical equivalence. Suppose $N$ has $2n$ integer boundary points
$v_0,\dots,v_{2n-1}$ in cclw order. Let $e_i$ be the edge joining $v_i$ to $v_{i+1}$;
it is antiparallel with $e_{n+i}$ because $P$ is centrally symmetric.
The proof now follows from the following lemma.
\end{proof}

\begin{lemma}
The strands of a minimal network 
are in bijection with, and have homology classes equal to,
the $n$ edge pairs $\{e_i,e_{i+n}\}$.
\end{lemma}

\begin{proof}
Since the network is minimal we can isotope the strands on the torus so that they
are geodesics for a Euclidean structure. Then we orient the strands so that
their $y$ coordinate is increasing, or if they have no $y$ component,
the $x$ component is increasing. 

Let $S_1,\dots,S_k$ be the strands, and let $e_i=(x_i,y_i)\in\Z^2$ 
be the homology class
of $S_i$.
The length $2k_i$ of $S_i$ 
(that is, the number of edges of the corresponding zig-zag path of $\G$)
is then determined by
$2k_i=\sum_j |e_i\wedge e_j|$ since exactly two strands cross at
each edge of $\G$.

Every strand gives a constraint on the homology class of a CRSF:
for any CRSF with homology class
$\omega$ we must have 
\be\label{constraints}|\omega\wedge e_i|\le k_i,\ee
since $|\omega\wedge e_i|$ is the intersection number of the CRSF
with the strand $S_i$. 

Suppse the $S_i$ are indexed in order of decreasing slope $x_i/y_i$.
Let $N_S$ be the convex polygon with side vectors 
$e_1,\dots,e_k,-e_1,\dots,-e_k$ in cclw order.
It has vertices $v_1,v_2,\dots,v_k,-v_1,-v_2,\dots,-v_k$ where
$$v_i=\frac12(e_1+\dots+e_{i-1}-e_i-\dots-e_k).$$
We claim that $N_S$ is the subset of $\omega\in\R^2$ which satisfy the constraints 
(\ref{constraints}).
To see this, note that if $v_i,v_{i+1}$ are consecutive vertices of $N_S$ and $e_i$
is the vector between them, then
\begin{eqnarray*}
2v_i\wedge e_i&=&e_1\wedge e_i+\dots+e_{i-1}\wedge e_i-e_{i+1}\wedge e_i-
\dots-e_k\wedge e_i\\
&=&\sum_j |e_j\wedge e_i|\\
&\le& 2k_i
\end{eqnarray*}
and similarly for $v_{i+1}$.
Thus $v_i,v_{i+1}$ satisfy the maximal possible constraint.

Thus $N\subset N_S$.
To see that $N=N_S$, is suffices to construct a CRSF with homology
class equal to each of the vertices $v_i$ of $N_S$.
To construct a set of cycles with homology class
$v_1=\frac12(-e_1-e_2-\dots-e_k)$, for example,
put weight $-1/2$ on each edge of the zig-zag path of $S_i$ (in the orientation
defined above).
Each edge of $\G$ will then have weight $0$ or $\pm1$;
since we assumed that the strands were drawn geodesically, at each vertex
there is exactly one incoming and one outgoing edge.
The oriented edges with weight $1$ form a set of vertex disjoint cycles
of the homology class $\omega$. 
\end{proof}

\subsection{Harnack curves and characterization of spectral curves}\label{Harnackdef}
The Laurent polynomial $P=\det\Delta$ has the following properties:
\begin{enumerate}
\item $P$ has real coefficients.
\item $P$ is symmetric: $P(z_1,z_2)=P(z_1^{-1},z_2^{-1})$.
\item $P(1,1)=0$.
\item $\{P=0\}$ is a Harnack curve.
\end{enumerate}

What we call Harnack curves were studied by Harnack and called simple Harnack curves
by Mikhalkin in \cite{M}.
The definition given by Mikhalkin in \cite{M} involves a topological condition on the real locus;
in \cite{MR} a more concise characterization was given:
a curve is Harnack if it intersects each torus
$\{|z|=c_1, |w|=c_2\}$ in at most two points (which are necessarily
complex conjugates if not real). In particular the map 
$(z,w)\mapsto(\log|z|,\log|w|)$ which maps the curve into $\R^2$
is generically $2$ to $1$. The image of $P=0$ under this
map is the {\bf amoeba} of $P$. See Figure \ref{amoeba} for the
amoeba of the example above.
\begin{figure}
\center{\includegraphics[width=6cm]{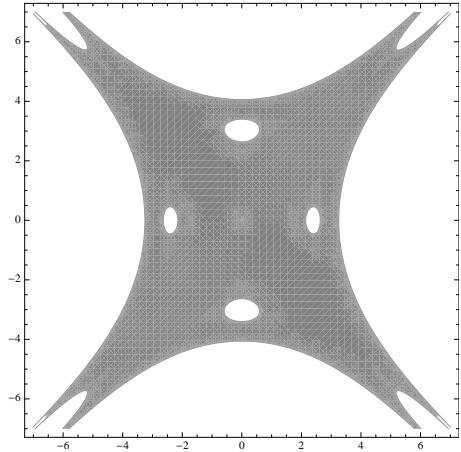}}
\caption{\label{amoeba}The amoeba (shaded) of the curve of Figure 
\protect{\ref{2X2torus}} 
above.
The boundary of the amoeba is the real locus of $P=0$; a model of $P=0$
can be obtained by gluing two copies of the shaded region together
along their common boundary.}
\end{figure}

Remarkably, the four properties above characterize the characteristic polynomials
of networks on a torus:
\begin{thm}\cite{KO,GK}\label{preimage}
For any polynomial $P$ satisfying the above properties there is a network 
$\G$ whose characteristic polynomial is $P$. The set of minimal networks 
with polynomial $P$, modulo electrical transformations, 
is a real torus of dimension
$g/2,$ where $g$ is the geometric genus (the number of real ovals of the amoeba of $P$).
\end{thm}

\section{Other surfaces}

For surfaces other than the annulus and torus,
the fundamental group will be nonabelian so 
it is appropriate to use an $\SL$-connection on $\G$ rather than a $\C^*$-connection.

Although we understand very little about this situation,
there is one 
positive result. Considering the Laplacian determinant 
as a function on the moduli space of flat
$\SL$-connections, it can be written in a unique way as a sum over CRSFs
according to the homotopy types of their cycles.
Let us explain this. We define a {\bf finite lamination} on $\Sigma$ to be an
isotopy class of finite 
pairwise disjoint collection of simple closed curves, none bounding
a disk. For example on an annulus there is a finite lamination 
for each nonnegative integer $k$, $k$ being
the number of its cycles. 

The cycles in a CRSF form a finite lamination. 
One can ask the question: for a given finite lamination $\cal L$,
what is the weighted sum of CRSFs having cycle set of type ${\cal L}$?
We have
\begin{thm}[\cite{Kenyon.bundle,FG}]\label{generalCRSF}
$P=\det\Delta$ can be written as
$$P=\sum_{\cal L}C_{\cal L}$$
where the sum is over all finite laminations and
$C_{\cal L}$ is the weighted sum (partition sum) of CRSFs
of cycle type ${\cal L}$. The coefficients $C_{\cal L}$ are functions of $P$ only,
and can be extracted
via an appropriate integration of $P$ over the representation variety 
${\text Hom}(\pi_1(\Sigma),\SL)$.
\end{thm}

As an example, take $\Sigma$ to be a pair of pants, that is,
a sphere minus three disks. Let $\G$ be a network embedded on $\Sigma$.
Take a flat $\SL$-connection on $\G$
with monodromy $A,B,C$ around the 
three boundary holes, where $ABC=1$. 

Then 
$$P=\det\Delta=\sum_{i,j,k\ge0}c_{i,j,k}(2-\Tr A)^i(2-\Tr B)^j(2-\Tr(AB))^k.$$
Given variables $X,Y,Z$ one can choose matrices $A,B$ in $\SL$ such that
$2-\Tr A=X, 2-\Tr B=Y,2-\Tr AB=Z$. 
Then 
$$P=P(X,Y,Z)=\sum_{i,j,k\ge0}c_{i,j,k}X^iY^jZ^k.$$
One can extract from this expression the coefficients $c_{i,j,k}$.

\old{
\begin{conj}
The topological equivalence class and conductances
of a minimal network on a surface can be recovered from
$\det\Delta$ as a function on the moduli space of flat 
$\SL$-connections on $\Sigma$. 
\end{conj}
}

\end{document}